\documentclass[12pt, reqno, a4paper,oneside]{amsart}
\usepackage{graphicx}
\usepackage{subfig}
\usepackage{color}
\usepackage{xcolor}
\usepackage[linesnumbered,ruled,vlined]{algorithm2e}
\usepackage{subcaption}
\usepackage{hyperref}
\usepackage{booktabs}
\usepackage[top=2.5cm,bottom=2.0cm,left=2.0cm,right=2.0cm]{geometry}
\usepackage{amsfonts, amsmath, amssymb, amsbsy, amsthm}
\usepackage{environments}
\usepackage{mathrsfs}
\usepackage{array}
\usepackage{tabularx} 
\numberwithin{theorem}{section}
\numberwithin{equation}{section}

\renewcommand{\thetheorem}{\arabic{section}.\arabic{theorem}}
\renewcommand{\thedefinition}{\arabic{section}.\arabic{definition}}

\renewcommand{\arraystretch}{1.5}

\newcommand{\uu}{\boldsymbol{u}}
\newcommand{\xx}{\boldsymbol{x}}

\title[Discrete-Time RFM for PDEs]{A Discrete-Time Random Feature Method for Nonlinear Evolution Equations with Implicit-Explicit Runge--Kutta Time Stepping}

\author{Haoran Zhou and Zhaohui Fu and Yangshuai Wang and Xinlong Feng}

\begin{document}

\begin{abstract}
    We study a discrete-time random feature method for nonlinear, time-dependent partial differential equations. In contrast to continuous-time formulations that treat time as an additional input variable, the method advances the solution step by step, with each time level computed from previously available states. The spatial solution at each step is represented in the random feature trial space, and the time discretization is given by an implicit-explicit Runge--Kutta (IMEX-RK, 4 stages, third-order) scheme. After splitting the operator into linear and nonlinear parts, each stage admits a linear least-squares formulation, which avoids nonlinear least-squares solves. We also derive a global error estimate for the fully discrete method, separating the contributions of the stage-wise RFM approximation, perturbations in the least-squares coefficients, and the temporal discretization. Numerical experiments for the Allen--Cahn, Burgers, Korteweg--De Vries, and Cahn--Hilliard equations show relative $L^2$-errors of order $10^{-6}$ and convergence rates consistent with the third-order IMEX scheme. A comparison with an IMEX-PINN variant shows that the proposed method achieves higher accuracy at substantially lower computational cost.
\end{abstract}

\maketitle

\section{Introduction}
\label{sec:intro}

Numerical methods for partial differential equations (PDEs) based on machine learning have attracted growing interest. Representative examples include Physics-Informed Neural Networks (PINNs)~\cite{cai2021physics, karniadakis2021physics}, Deep Ritz~\cite{yu2018deep}, Weak Adversarial Networks~\cite{zang2020weak}, and random feature methods~\cite{chen2022bridging, chen2023random, chi2024random}. These methods benefit from flexible trial spaces and can handle complex geometries or data-driven settings; see also~\cite{chai2021deep, choudhary2022recent, voulodimos2018deep, young2018recent, otter2020survey, saha2021hierarchical} for reviews of machine learning in scientific computing.

Among these methods, the random feature method (RFM)~\cite{chen2022bridging, chen2023random, chi2024random} is computationally attractive because the hidden parameters of the neural trial space are fixed after random initialization, so the unknown coefficients enter linearly and can be determined by a least-squares solve. The approximation combines random feature basis functions with partition-of-unity (PoU) functions, a construction closely related to randomized neural networks and extreme learning machines~\cite{schmidt1992feed, zhang2020new, nakajima2021reservoir, dong2021local, wang2024extreme, li2019towards, scardapane2017randomness, wojcik2019training}. Since its introduction, RFM has been extended in several directions: adaptive feature and collocation point distributions for resolving local structures~\cite{deng2025adaptivefeaturecapturemethod}, Christoffel-function-based sampling for improved approximation quality~\cite{adcock2026}, randomized preconditioned iterative solvers for the resulting linear systems~\cite{TAN2026117255}, randomized Newton-type methods for nonlinear least-squares systems~\cite{tan2025}, asymptotic-preserving formulations for multiscale transport equations~\cite{CHEN2025114103}, weak-form extensions~\cite{kuvakin2025weakrandomfeaturemethod}, and quantum computing adaptations~\cite{hu2025quantumrandomfeaturemethod}. These advances, however, focus primarily on stationary or mildly time-dependent settings and do not address the two central difficulties that arise when RFM is applied to nonlinear evolution equations.

The first difficulty is temporal discretization. Much of the existing work on neural PDE solvers adopts a continuous-time formulation~\cite{raissi2019physics}, where time is treated as an additional input variable. This viewpoint does not enforce causal time marching and may become inefficient or inaccurate when the solution changes rapidly. Discrete-time formulations have been studied for PINNs and related models~\cite{raissi2019physics, jagtap2020extended, stiasny2021learning, moya2023dae}, but a corresponding framework for RFM remains limited. The second difficulty is nonlinearity. After time discretization, a direct RFM formulation for a nonlinear evolutionary PDE leads to a nonlinear least-squares problem, removing the linear structure that makes RFM computationally advantageous. This is particularly restrictive for strongly nonlinear equations such as phase-field models~\cite{chen2002phase, badalassi2003computation}, where standard nonlinear least-squares methods~\cite{gill1978algorithms} may be inefficient or unreliable.

Deng et al.~\cite{DengHe2025_PhysFluids} recently combined Runge--Kutta time discretization with RFM for multiphase flow problems. Their results show that classical time stepping can be integrated with RFM, but the scheme is restricted to second-order accuracy and the treatment of strongly nonlinear terms remains limited.

The present work addresses both difficulties by combining RFM in space with an implicit-explicit (IMEX) Runge--Kutta time discretization~\cite{ascher1995implicit}. The method advances the solution one step at a time: the stiff linear part is handled implicitly, while the nonlinear part is treated explicitly. As a result, the unknown stage solution remains linear in the RFM coefficients, so each stage is computed from a linear least-squares problem. This distinguishes the proposed method from continuous-time RFM formulations on the full space-time domain and from discrete-time neural solvers that require repeated nonlinear optimization or transfer learning across time steps.

The main contributions of this work are as follows:
\begin{itemize}
    \item We formulate a discrete-time RFM scheme for time-dependent PDEs that advances the solution sequentially in time, in contrast to continuous-time space-time formulations.
    \item We combine this scheme with an IMEX Runge--Kutta discretization so that each stage of the nonlinear problem reduces to a linear least-squares solve.
    \item We derive a global error bound for the fully discrete method that separates the contributions of the stage-wise RFM approximation, perturbations in the least-squares coefficients, and the temporal discretization. Numerical results for the Allen--Cahn, Burgers, Korteweg--De Vries, and Cahn--Hilliard equations show relative \(L^2\)-errors of order $10^{-6}$ and overall third-order convergence over the range of step sizes tested.
\end{itemize}

The remainder of the paper is organized as follows. Section~\ref{sec:rfm} reviews RFM for static problems and introduces the discrete-time formulation for time-dependent PDEs. Section~\ref{sec:IMEX} presents the IMEX-RK algorithm and the error analysis. Section~\ref{sec:numer} reports the numerical experiments. Section~\ref{sec:conclusion} concludes the paper.

\section{Random Feature Method (RFM)}
\label{sec:rfm}

This section establishes the notation for RFM and introduces the discrete-time formulation used later. Section~\ref{sec:sub:static} recalls the spatial trial space and least-squares formulation for stationary PDEs. Section~\ref{sec:sub:time-dep} extends the notation to evolution equations and uses Crank--Nicolson as an illustrative example of a discrete-time RFM construction. The IMEX-RK method used in the rest of the paper is introduced in Section~\ref{sec:IMEX}.

\subsection{RFM for Static Problems}
\label{sec:sub:static}

Let $\Omega \subset \mathbb{R}^d$ be a bounded, connected domain with dimension $d \in \mathbb{N}^+$, and let $\xx=(x_1,\ldots,x_d)^{T}\in\Omega$. We consider the stationary problem of finding a possibly vector-valued solution $\uu(\xx)\in\mathbb{R}^{d_u}$ such that
\begin{equation}
\begin{cases}
\mathcal{L} \uu(\xx) = \boldsymbol{f}(\xx), & \xx \in \Omega, \\
\mathcal{B} \uu(\xx) = \boldsymbol{g}(\xx), & \xx \in \partial \Omega,
\end{cases}
\end{equation}
where $\mathcal{L}$ is a differential operator and $\mathcal{B}$ is a boundary operator. We use boldface to denote vector-valued quantities.

The domain decomposition underlying the RFM approximation~\cite{chan1994domain, smith1997domain} is defined as follows. Let
$\widehat{\Omega} = [a_1, b_1] \times [a_2, b_2] \times \ldots \times [a_d, b_d] =: \prod_{i=1}^d [a_i, b_i]$ 
be a hyper-rectangle containing $\Omega$. We partition $\widehat{\Omega}$ into $M$ non-overlapping hyper-rectangles,
\[
\widehat{\Omega} = \bigcup_{i=1}^M \Omega_i := \bigcup_{i=1}^M \prod_{j=1}^d [a^{(i)}_{j}, b^{(i)}_{j}],
\]
where each $\Omega_i$ is a subdomain. This decomposition plays the same role as an element partition in finite element constructions~\cite{reddy1993introduction}.

On each subdomain, we map $\Omega_i$ to the reference cube $[-1,1]^d$. For each $i=1,\ldots,M$, define
\[
\widetilde{\xx}^{(i)}=(\widetilde{x}^{(i)}_{1},\ldots,\widetilde{x}^{(i)}_{d})^T
\]
by
\[
\widetilde{\xx}^{(i)} = \frac{\xx - \boldsymbol{c}^{(i)}}{\boldsymbol{r}^{(i)}},
\]
where $\boldsymbol{c}^{(i)}$ and $\boldsymbol{r}^{(i)}$ are the center and half-width of $\Omega_i$,
\[
\boldsymbol{c}^{(i)} = \left( \frac{b^{(i)}_{1} + a^{(i)}_{1}}{2}, \ldots, \frac{b^{(i)}_{d} + a^{(i)}_{d}}{2} \right)^T \qquad\text{and}\qquad
\boldsymbol{r}^{(i)} = \left( \frac{b^{(i)}_{1} - a^{(i)}_{1}}{2}, \ldots, \frac{b^{(i)}_{d} - a^{(i)}_{d}}{2} \right)^T.
\]
This normalization is used to define local random features.

For each $\Omega_i$, RFM uses $J_n \in \mathbb{N}^+$ random features $\{\phi^{(i)}_{j}(\boldsymbol{x})\}_{1 \leq i \leq M, 1 \leq j \leq J_n}$ defined by a two-layer neural network with fixed random parameters $\boldsymbol{k}^{(i)}_{j} \in \mathbb{R}^d$ and $b^{(i)}_{j} \in \mathbb{R}$:
\begin{equation}\label{eq:phi_ij}
    \phi^{(i)}_{j}(\xx) = \sigma\big(\boldsymbol{k}^{(i)}_{j} \cdot \widetilde{\xx}^{(i)} + b^{(i)}_{j}\big), 
\end{equation}
where $\sigma$ is a nonlinear scalar activation function, typically chosen as $\tanh$ or a trigonometric function~\cite{chen2022bridging}. The components of $\boldsymbol{k}^{(i)}_{j}$ and $b^{(i)}_{j}$ are sampled independently from $\mathcal{U}([-R_m, R_m])$, where $R_m \in \mathbb{R}^+$ controls the magnitude of the random parameters. These parameters are fixed throughout the coefficient solve. 


The construction in~\eqref{eq:phi_ij} extends to multi-layer randomized neural networks by composing the mapping recursively over several layers, consistent with the framework of~\cite{schmidt1992feed}. In what follows, we restrict attention to the standard single-hidden-layer RFM.

To assemble the local features into a global trial space, we introduce partition-of-unity (PoU) functions on the reference cube. In one dimension, two commonly used choices are
\begin{align*}
\psi_{\rm a}(x) &= \mathbb{I}_{[-1, 1]}(x), \\
\psi_{\rm b}(x) &= \mathbb{I}_{[-\frac{5}{4}, -\frac{3}{4}]}(x) \frac{1 + \tanh(2\pi x)}{2} + \mathbb{I}_{[-\frac{3}{4}, \frac{3}{4}]}(x) + \mathbb{I}_{[\frac{3}{4}, \frac{5}{4}]}(x) \frac{1 - \tanh (2\pi x)}{2},
\end{align*}
where $\mathbb{I}$ denotes the characteristic function. The function $\psi_{\mathrm{a}}$ is discontinuous, while $\psi_{\mathrm{b}}$ is continuously differentiable. In higher dimensions, $\psi$ is defined by the tensor product $\psi(\boldsymbol{x}) = \prod_{i=1}^{d} \psi(x_i)$ for $\boldsymbol{x} = (x_1, \ldots, x_d)$. The PoU on each subdomain is obtained by scaling this reference function to $\Omega_i$. 

As illustrated in Figure~\ref{figs:RFM-network}, the proposed method follows a discrete-time framework. The spatial domain is first decomposed into subdomains, on each of which a local random feature representation is constructed. These local approximations are then assembled into a global trial function through the partition-of-unity mechanism. Based on this spatial representation, the solution is advanced from one time level to the next by a stage-wise IMEX-RK procedure.

\begin{figure}[htbp]
    \centering
    \includegraphics[height=7.5cm]{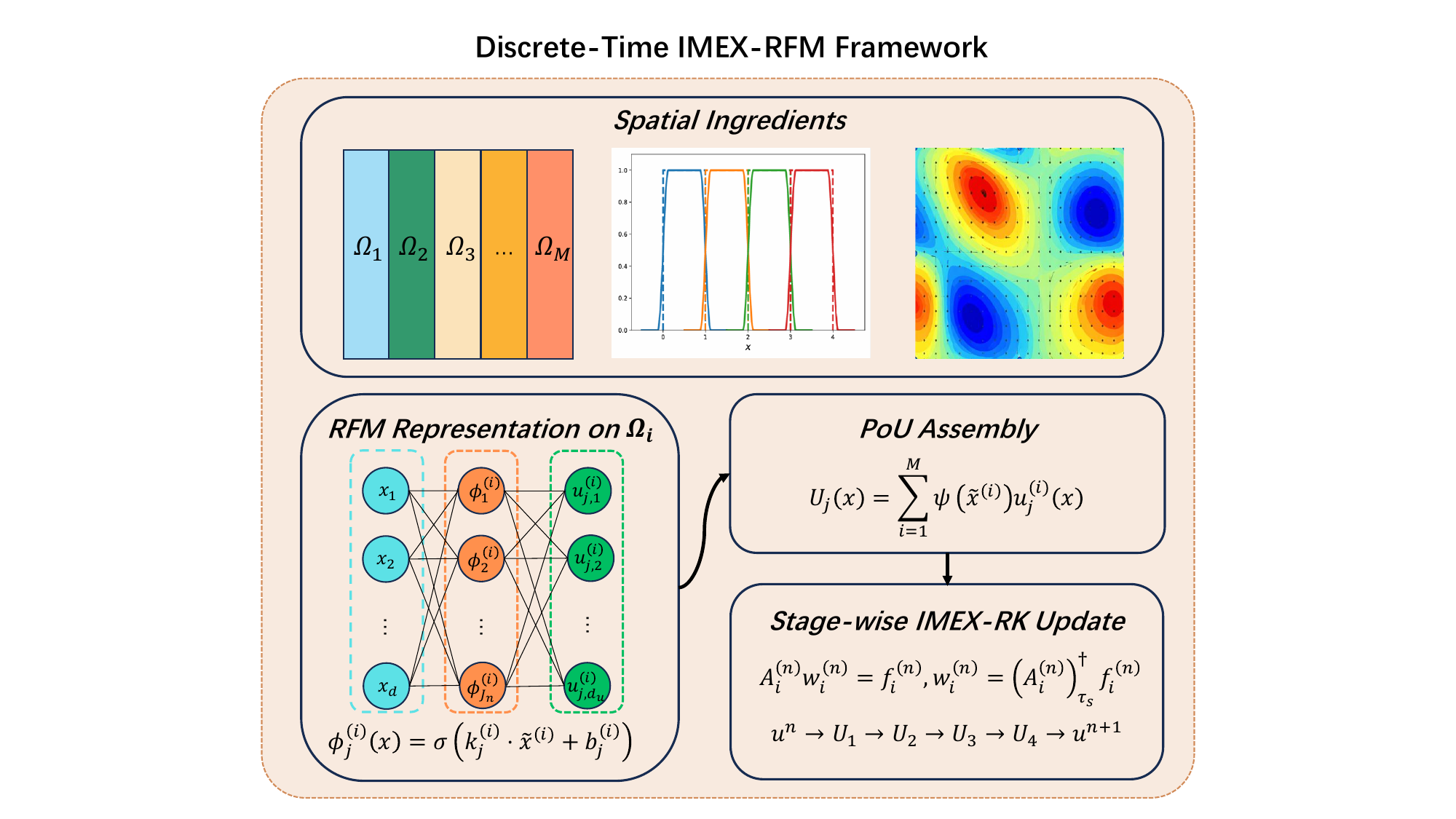}
    \caption{Illustration of the proposed discrete-time IMEX-RFM method. Local random feature approximations are constructed on subdomains, assembled by the partition-of-unity mechanism, and then updated through stage-wise IMEX-RK time stepping. }
    \label{figs:RFM-network}
\end{figure}

The global RFM trial function is then
\begin{equation}\label{eq:u_rfm}
    u_{N}(\xx; \boldsymbol{\omega}) = \sum_{i=1}^M \psi(\widetilde{\xx}^{(i)}) u^{(i)}(\xx) = \sum_{i=1}^M \psi(\widetilde{\xx}^{(i)}) \sum_{j=1}^{J_n} \omega_{ij} \phi^{(i)}_{j}(\xx),
\end{equation}
where $\omega_{ij}\in\mathbb{R}$ are unknown coefficients and $N=MJ_n$ is the total number of degrees of freedom. For vector-valued solutions with dimension $d_u$, each component is approximated by~\eqref{eq:u_rfm}, that is,
\begin{equation}\label{eq:uN}
\uu_{N}(\xx; \boldsymbol{\omega}) = \big(u_N^{1}(\xx; \boldsymbol{\omega}), \ldots, u^{d_u}_{N}(\xx; \boldsymbol{\omega})\big)^{T}.
\end{equation}

To determine the coefficients, we enforce the PDE and boundary conditions at collocation points and solve a weighted least-squares problem. Let $\xx^{(i)}_{1},\ldots,\xx^{(i)}_{Q}$ be $Q\in\mathbb{N}^+$ collocation points in each $\Omega_i$. Introduce positive weights $\lambda^{\rm p}_{i,k,q}>0$ and $\lambda^{\rm b}_{i,k,e}>0$ for $1\le i\le M$, $1\le k\le d_u$, and $1\le q,e\le Q$ (with $\xx_{i,e}\in\partial\Omega$), and define
$\Lambda^{\rm p}_{i,q}=\mathrm{diag}(\lambda^{\rm p}_{i,1,q},\ldots,\lambda^{\rm p}_{i,d_u,q})$
and
$\Lambda^{\rm b}_{i,e}=\mathrm{diag}(\lambda^{\rm b}_{i,1,e},\ldots,\lambda^{\rm b}_{i,d_u,e})$.
The coefficients are obtained by minimizing
\begin{equation}\label{eq:loss}
\begin{aligned}
\mathrm{Loss}(\boldsymbol{\omega}) = 
\sum_{i=1}^M \Bigg( 
&\sum_{q=1}^Q 
\left\| \Lambda^{\rm p}_{i,q} 
\left( \mathcal{L} \boldsymbol{u}_N(\xx^{(i)}_{q}; \boldsymbol{\omega}) - \boldsymbol{f}(\xx^{(i)}_{q}) \right) \right\|_2^2 \\
&+ \sum_{\xx_{i,e} \in \partial \Omega} 
\left\| \Lambda^{\rm b}_{i,e} 
\left( \mathcal{B} \boldsymbol{u}_N(\xx_{i,e}; \boldsymbol{\omega}) - \boldsymbol{g}(\xx_{i,e}) \right) \right\|_2^2 
\Bigg),
\end{aligned}
\end{equation}
where $\boldsymbol{u}_N(\xx, \boldsymbol{\omega})$ is of the form~\eqref{eq:uN}. Equivalently,
\begin{equation}\label{eq:least-squares}
\boldsymbol{\widehat{\omega}} = \min_{\boldsymbol{\omega}} \| \mathbf{A}\boldsymbol{\omega} - \mathbf{b} \|^2,
\end{equation}
where $\mathbf{A}$ is the design matrix and $\mathbf{b}$ is the observation vector. The detailed construction of these components is given in~\cite[Section 2.2]{chen2023random}.
    
If the discontinuous PoU $\psi_{\mathrm{a}}$ is used, continuity conditions between adjacent subdomains must be enforced by adding regularization terms to~\eqref{eq:loss}~\cite{chen2022bridging}. If $\psi_{\mathrm{b}}$ is used, no such regularization is required provided that $\mathcal{L}$ is a second-order differentiable operator.

\begin{remark}
When $\mathcal{L}$ and $\mathcal{B}$ are linear, \eqref{eq:loss} is a standard linear least-squares problem and can be solved by classical algorithms~\cite{chen2024optimization}. For nonlinear evolutionary problems, a direct discrete-time formulation generally leads to nonlinear least-squares systems. Section~\ref{sec:sub:time-dep} introduces the time-dependent setting, which is combined with IMEX time stepping in Section~\ref{sec:IMEX}.
\end{remark}

\subsection{RFM for Time-Dependent PDEs}
\label{sec:sub:time-dep}

Given a final time $T > 0$, consider
\begin{equation}\label{eq:timePDEs}
\begin{aligned}
    \boldsymbol{u}_t(\boldsymbol{x}, t) &= \mathcal{N}\big[\boldsymbol{u}(\boldsymbol{x}, t)\big], && \quad (\boldsymbol{x}, t) \in \Omega \times [0, T], \\
    \mathcal{B}\boldsymbol{u}(\xx, t) &= \boldsymbol{g}(\xx), && \quad (\xx, t) \in \partial \Omega \times [0, T], \\
    \mathcal{I}\boldsymbol{u}(\xx, 0) &= \boldsymbol{h}(\xx), && \quad  \xx \in \Omega,
\end{aligned}
\end{equation}
where $\boldsymbol{g}$ and $\boldsymbol{h}$ are given functions and $\mathcal{I}$ is the operator associated with the initial condition.

Two common formulations exist. In the continuous-time approach, time is treated as another input variable and the approximation is built on the full space-time domain. In the discrete-time approach, time is discretized first and one solves a sequence of spatial problems. The latter makes temporal causality explicit, since the approximation at time level $n+1$ is computed from previously available states. The present paper adopts the discrete-time viewpoint for both computational and modeling reasons. Computationally, many neural PDE solvers, including PINN-type methods, rely on SGD or Adam to handle nonlinear optimization, and their accuracy and efficiency are often constrained by the training process itself. By separating time advancement from spatial approximation, our method reduces each update to a linear least-squares problem, which allows the direct use of classical linear solvers and avoids repeated nonlinear optimization. From the viewpoint of the underlying PDE, moreover, time and space are not on an equal footing in evolution problems. A step-by-step treatment in time, combined with a spatial approximation at each time level, therefore better reflects the causal structure and physical meaning of the dynamics.
To illustrate how a discrete-time RFM formulation is constructed, we use the Crank--Nicolson scheme as an example; the IMEX-RK algorithm employed in the paper is presented in Section~\ref{sec:IMEX}. Given a time step $\Delta t$, consider first the scalar case, with the vector-valued case treated componentwise. Suppose $u^n_N(\xx)$ approximates the exact solution $u(\xx, n \Delta t)$ of \eqref{eq:timePDEs}. The Crank--Nicolson update reads
\begin{equation}\label{eq:crank_nicolson}
     \frac{u^{n+1}_N(\xx) - u^n_N(\xx)}{\Delta t} = \mathcal{N} \left[ \frac{u^{n+1}_N(\xx) + u^n_N(\xx)}{2} \right].
\end{equation}
With the spatial RFM approximation from Section~\ref{sec:sub:static}, we represent $u^{n+1}_N(\xx)$ by $u_N(\xx; \boldsymbol{\omega}^{n+1})$, assuming that $u_{N}(\xx; \boldsymbol{\omega}^n)$ is already known. The coefficients at the next step are then obtained by minimizing
\begin{equation}\label{eq:loss_function_discrete}
\begin{aligned}
\mathrm{Loss}^{n+1}(\boldsymbol{\omega}^{n+1}) =& 
\sum_{i=1}^M \sum_{\xx_{i,e} \in \partial \Omega} 
\left\| \Lambda^{\rm b}_{i,e} 
\left( \mathcal{B} u_N(\xx_{i,e}; \boldsymbol{\omega}^{n+1}) - \boldsymbol{g}(\xx_{i,e}) \right) \right\|_2^2 \\
&+ \sum_{i=1}^M \sum_{q=1}^Q 
\left\| \Lambda^{\rm p}_{i,q} \left( 
\frac{u_N(\xx_q^{(i)}; \boldsymbol{\omega}^{n+1}) - u_N(\xx_q^{(i)}; \boldsymbol{\omega}^{n})}{\Delta t} 
\right. \right. \\
&\quad \left. \left.
- \mathcal{N} \left[ \frac{u_N(\xx_q^{(i)}; \boldsymbol{\omega}^{n+1}) + u_N(\xx_q^{(i)}; \boldsymbol{\omega}^{n})}{2} \right] 
\right) \right\|_2^2.
\end{aligned}
\end{equation}
When the spatial approximation is replaced by a PINN ansatz, related constructions are often referred to as discrete PINNs~\cite{raissi2019physics}.

More generally, given a time-differencing operator $\mathfrak{D}$ and the corresponding spatial combination $\mathfrak{U}$, the next-step approximation is defined by
\begin{equation}\label{eq:general_discrete}
    \mathfrak{D}\big(x; \Delta t, u^n, u^{n+1}, \ldots, u^{n+k_{\rm d}}\big) = \mathcal{N} \big[\mathfrak{U}\big(x; u^n, u^{n+1}, \ldots, u^{n+k_{\rm u}}\big)\big],
\end{equation}
where $k_{\rm d}$ and $k_{\rm u}$ depend on the chosen time discretization. For Crank--Nicolson, $k_{\rm d}=k_{\rm u}=1$. The loss function corresponding to~\eqref{eq:general_discrete} is obtained from~\eqref{eq:loss_function_discrete} in the same way. In the remainder of the paper, we specialize this framework to an implicit-explicit (IMEX) scheme~\cite{ascher1995implicit}; the resulting algorithm is given in Section~\ref{sec:IMEX}.

The discrete-time viewpoint leads to the step-by-step update
\[
u_N(\xx; \boldsymbol{\omega}^0) 
\xrightarrow{\text{RFM}} 
u_N(\xx; \boldsymbol{\omega}^1) 
\xrightarrow{\text{RFM}} 
\cdots 
\xrightarrow{\text{RFM}} 
u_N(\xx; \boldsymbol{\omega}^n) 
\xrightarrow{\text{RFM}} 
u_N(\xx; \boldsymbol{\omega}^{n+1}) 
\xrightarrow{\text{RFM}} 
\cdots,
\]
where $u_N(\xx; \boldsymbol{\omega}^0)$ is obtained by fitting the initial condition.

\begin{remark}[Stability]
\label{re:stab}
The step-by-step solve of~\eqref{eq:general_discrete} relies on the accuracy of the spatial RFM approximation at each time level. This motivates the review of the static formulation in Section~\ref{sec:sub:static} and the error analysis in Section~\ref{sec:sub:error}.
\end{remark}

\begin{remark}[Linearity of the stage solve]
\label{re:eff}
For RFM, direct application of Crank--Nicolson or similar schemes to nonlinear problems produces nonlinear least-squares systems. The IMEX formulation in Section~\ref{sec:IMEX} avoids this by keeping each stage linear in the unknown coefficients, eliminating the need for iterative nonlinear solvers at every time step.
\end{remark}

\section{IMEX-RK(4,3) Based RFM Method and Error Analysis}
\label{sec:IMEX}

\subsection{Stage-wise Linear Systems}
\label{sec:sub:imex-rfm}

This subsection formulates the IMEX-RK(4,3) discretization in the RFM trial space and derives the stage-wise algebraic systems.

Following Section~\ref{sec:sub:time-dep}, we consider the evolution problem
\begin{equation}
\uu_t(\xx,t)=\mathcal{N}(\uu(\xx,t)), \qquad t\in(0,T], \qquad \uu(\xx,0)=\uu^0(\xx),
\label{eq:3.1}
\end{equation}
where \(\uu(\xx,t)\in\mathbb{R}^{d_u}\). We split the operator as
\begin{equation}
\mathcal{N}(\uu)=\mathcal{L}(\uu)+\mathcal{G}(\uu),
\label{eq:3.2}
\end{equation}
where \(\mathcal{L}\) is the stiff linear part and \(\mathcal{G}\) is the nonlinear part.

In each stage \(j\), we write
\begin{equation}
U_j(\xx)
=
\sum_{i=1}^{M}\psi\big(\widetilde{\xx}^{(i)}\big)\sum_{k=1}^{J_n}\omega_{ik}^{(j)}\phi_k^{(i)}(\xx),
\label{eq:3.3c}
\end{equation}
where \(U_j\) is the \(j\)-th stage solution in the current IMEX-RK step.

As a baseline, the first-order IMEX step is
\begin{equation}
\frac{\uu^{n+1}-\uu^n}{\Delta t}
=
\mathcal{L}(\uu^{n+1})+\mathcal{G}(\uu^n).
\label{eq:3.4}
\end{equation}

We now use an \(s\)-stage IMEX-RK scheme with coefficients \(\{a_{ij}\}\), \(\{\hat{a}_{ij}\}\), \(\{b_i\}\), and \(\{\hat{b}_i\}\). The corresponding Butcher tableau is listed in Appendix Table~\ref{tab:imex-rk-coeff}. The stage equations are
\begin{equation}
U_i
=
\uu^n
+
\Delta t \sum_{j=1}^{i} a_{ij} \mathcal{L}(U_j)
+
\Delta t \sum_{j=1}^{i-1} \hat{a}_{ij} \mathcal{G}(U_j),
\qquad i=1,\dots,s,
\label{eq:3.5}
\end{equation}
and the update is
\begin{equation}
\uu^{n+1}
=
\uu^n
+
\Delta t \sum_{i=1}^{s} b_i \mathcal{L}(U_i)
+
\Delta t \sum_{i=1}^{s} \hat{b}_i \mathcal{G}(U_i).
\label{eq:3.6}
\end{equation}
Separating the current-stage term gives
\begin{equation}
\left(I-\Delta t\,a_{ii}\mathcal{L}\right)U_i
=
\uu^n
+
\Delta t \sum_{j=1}^{i-1} a_{ij}\mathcal{L}(U_j)
+
\Delta t \sum_{j=1}^{i-1} \hat{a}_{ij}\mathcal{G}(U_j),
\label{eq:3.7}
\end{equation}
for \(i=1,\dots,s\). This form makes the stage structure explicit: the unknown at stage \(i\) is only \(U_i\), and all terms involving \(U_j\) with \(j<i\) have already been computed and enter the right-hand side as known quantities. Since \(\mathcal{L}\) is linear and treated implicitly, the left-hand side depends linearly on \(U_i\). Since \(\mathcal{G}\) is treated explicitly, it is never evaluated at the current unknown stage. Consequently, each stage equation is linear in \(U_i\).

For the four-stage scheme used in the experiments, the stage equations are
\begin{equation}
\left(I-\Delta t\,a_{11}\mathcal{L}\right)U_1=\uu^n,
\label{eq:3.8a}
\end{equation}
\begin{equation}
\left(I-\Delta t\,a_{22}\mathcal{L}\right)U_2
=
\uu^n
+
\Delta t\,a_{21}\mathcal{L}(U_1)
+
\Delta t\,\hat{a}_{21}\mathcal{G}(U_1),
\label{eq:3.8b}
\end{equation}
\begin{equation}
\left(I-\Delta t\,a_{33}\mathcal{L}\right)U_3
=
\uu^n
+
\Delta t\left[a_{31}\mathcal{L}(U_1)+a_{32}\mathcal{L}(U_2)\right]
+
\Delta t\left[\hat{a}_{31}\mathcal{G}(U_1)+\hat{a}_{32}\mathcal{G}(U_2)\right],
\label{eq:3.8c}
\end{equation}
\begin{equation}
\begin{aligned}
\left(I-\Delta t\,a_{44}\mathcal{L}\right)U_4
&=
\uu^n
+
\Delta t\left[a_{41}\mathcal{L}(U_1)+a_{42}\mathcal{L}(U_2)+a_{43}\mathcal{L}(U_3)\right] \\
&\quad
+
\Delta t\left[\hat{a}_{41}\mathcal{G}(U_1)+\hat{a}_{42}\mathcal{G}(U_2)+\hat{a}_{43}\mathcal{G}(U_3)\right],
\end{aligned}
\label{eq:3.8d}
\end{equation}
and
\begin{equation}
\uu^{n+1}
=
\uu^n
+
\Delta t \sum_{i=1}^{4} b_i \mathcal{L}(U_i)
+
\Delta t \sum_{i=1}^{4} \hat{b}_i \mathcal{G}(U_i).
\label{eq:3.8e}
\end{equation}

To discretize these stage equations in space, we represent each stage solution in the RFM trial space. For each stage \(i\), we write
\begin{equation}
\widetilde{U}_i(\xx)
=
\sum_{m=1}^{M}\psi\big(\widetilde{\xx}^{(m)}\big)\sum_{k=1}^{J_n}\omega_{mk}^{(i)}\phi_k^{(m)}(\xx),
\label{eq:3.9}
\end{equation}
where \(\omega_{mk}^{(i)}\) are the unknown coefficients collected in the stage vector \(w_i^{(n)}\).

The linearity of the stage solve is preserved after this substitution. The representation \eqref{eq:3.9} is linear in \(\omega_{mk}^{(i)}\), and \(\mathcal{L}(\widetilde{U}_i)\) is also linear in \(\omega_{mk}^{(i)}\) because \(\mathcal{L}\) is linear. Hence the left-hand side of \eqref{eq:3.7}, evaluated with \(\widetilde{U}_i\), is linear in the unknown stage coefficients. On the right-hand side, all terms depend only on previously computed stages \(U_j\), \(j<i\). After collocation, each stage reduces to a linear least-squares system in coefficient space:
\begin{equation}
A_i^{(n)} w_i^{(n)} = f_i^{(n)}, \qquad i=1,\dots,s,
\label{eq:3.10}
\end{equation}
where \(w_i^{(n)}\) collects the coefficients in the stage expansion \eqref{eq:3.9}. In practice, the stage coefficients are computed by the truncated-SVD pseudoinverse:
\begin{equation}
w_i^{(n)}=\left(A_i^{(n)}\right)_{\tau_s}^{\dagger} f_i^{(n)},
\label{eq:3.11}
\end{equation}
where \(\tau_s\) is the singular-value cutoff. This regularizes the solve when the least-squares matrix is ill-conditioned.

Once the stage coefficients are available, \(\uu^{n+1}\) is obtained from \eqref{eq:3.8e}. The time update takes the form
\[
\uu^0 \longrightarrow \{U_i^0\}_{i=1}^4 \longrightarrow \uu^1 \longrightarrow \cdots \longrightarrow \uu^K.
\]
At each time level, the same loop of stage solve and time update is repeated. 

The full procedure is summarized in Algorithm~\ref{alg:imex-rfm}.

\begin{algorithm}[H]
\caption{IMEX-RK(4,3) Based RFM Algorithm}
\label{alg:imex-rfm}
\KwIn{
Number of subdomains \(M\); number of random feature functions on each subdomain \(J_n\); number of collocation points per subdomain \(Q\); terminal time \(T\); number of time steps \(K\); IMEX-RK(4,3) coefficients \(a_{ij}\), \(\hat{a}_{ij}\), \(b_i\), and \(\hat{b}_i\); range of uniform initialization \(R_m\); singular-value cutoff \(\tau_s\).
}
\KwOut{
Approximate solutions \(\uu^1,\dots,\uu^K\).
}
Divide \(\Omega\) into \(M\) non-overlapping subdomains\;
Sample \(Q\) collocation points in each subdomain\;
Construct the random feature basis on each subdomain\;
Fit the initial condition to obtain \(\uu^0\)\;
\For{$n=0,1,2,\ldots,K-1$}{
    \For{$i=1,2,3,4$}{
        Assemble the stage system \(A_i^{(n)}w_i^{(n)}=f_i^{(n)}\) from \eqref{eq:3.8a}--\eqref{eq:3.8d}\;
        Solve for \(w_i^{(n)}\) by the truncated-SVD pseudoinverse \eqref{eq:3.11}\;
        Construct the stage solution \(U_i\) using the RFM approximation \eqref{eq:3.9}\;
    }
    Update \(\uu^{n+1}\) using \eqref{eq:3.8e}\;
}
\Return \(\uu^1,\dots,\uu^K\)\;
\end{algorithm}
\begin{remark}[Treatment of the explicit nonlinear term] 
For the numerical examples considered below, we distinguish two forms of the explicit nonlinear term: $\mathcal{G}(\uu)=\mathcal{K}(\uu)$ and $\mathcal{G}(\uu)=\mathcal{F}\mathcal{K}(\uu)$, where $\mathcal{F}$ is a spatial differential operator. We call the former Type-I and the latter Type-II. For Type-I problems, $\mathcal{G}(U_j)$ is evaluated pointwise once $U_j$ is known. For Type-II problems, the differential operator is applied after the stage solution has been represented in the RFM form~\eqref{eq:3.3c}. In this case, $\mathcal{F}\mathcal{K}(U_j)$ is evaluated from the above representation by automatic differentiation during assembly.

\end{remark}

\subsection{Error Analysis}
\label{sec:sub:error}

This subsection derives the error estimate for the IMEX-RK/RFM method. The argument separates three ingredients: the approximation property of the RFM trial space, perturbations introduced by the numerical least-squares solve, and the time-discretization error of the IMEX scheme.

\begin{lemma}\label{lem:universal}
Let \( \sigma \) be a bounded, nonconstant, piecewise continuous activation function, and let \(G_N(\xx)\) be a scalar function of the form \eqref{eq:u_rfm}, where \(N=MJ_n\) denotes the number of degrees of freedom. Then, for any \(f\in C(\Omega)\), there exists a sequence of coefficient vectors such that
\[
\lim_{N\to\infty}\|G_N - f\|_{L^2(\Omega)}=0.
\]
\end{lemma}

\begin{proof}
Since $\sigma$ is bounded, nonconstant, and piecewise continuous, the conditions of the universal approximation theorem (Theorem~2.1 in~\cite{1650244}) are satisfied. The conclusion follows by taking $N\to\infty$.
\end{proof}

We next bound the error in the coefficient vectors. Let \(\delta U\) denote the difference between the numerical coefficient vector \(\widehat{U}\) and the optimal coefficient vector \(\widetilde{U}\) on a subdomain \(\Omega_n\). Let \(\{b_j\}_{j=1}^{J_n}\) be an orthonormal basis of \(\mathbb{R}^{J_n}\). Then there exist \(\{\delta u_j\}_{j=1}^{J_n}\) such that
\[
\delta U=\sum_{j=1}^{J_n}\delta u_j b_j,
\]
where \(\delta u_j\), \(j=1,\dots,J_n\), are independent and identically distributed random variables satisfying
\begin{equation}
E[\delta u_j]=0,\qquad E[|\delta u_j|]=\mu>0,\qquad E[\delta u_j^2]=\delta^2>0.
\label{exp:error estimate}
\end{equation}

\begin{lemma}\label{lem:coeff_error}
For the RFM approximation, let \(\widehat{U}=(\hat{u}^{i}_{nj})^\top\) and \(\widetilde{U}=(\tilde{u}^{i}_{nj})^\top\) denote the numerical coefficient vector and the optimal coefficient vector obtained by minimizing the loss function \eqref{eq:loss}, respectively. Then
\[
E_{\delta u}\left[\|\widehat{U}-\widetilde{U}\|_{2}\right]
\le
\sqrt{d_uMJ_n}\,\delta,
\]
where \(\delta\) is defined in \eqref{exp:error estimate}.
\end{lemma}

\begin{proof}
By Jensen's inequality,
\[
\begin{aligned}
\left(E_{\delta u}\left[\|\widehat{U}-\widetilde{U}\|_{2}\right]\right)^2
&\le
E_{\delta u}\left[\|\widehat{U}-\widetilde{U}\|_{2}^2\right] \\
&=
E_{\delta u}\left[
\sum_{i=1}^{d_u}\sum_{n=1}^{M}
\left\|
\sum_{j=1}^{J_n}\delta u^{i}_{nj} b_j
\right\|_2^2
\right] \\
&=
d_uMJ_n\delta^2.
\end{aligned}
\]
The last equality uses the orthonormality of \(\{b_j\}\) and the independence of \(\delta u^{i}_{nj}\). Taking square roots completes the proof.
\end{proof}

The following lemma transfers coefficient perturbations to function-space errors.

\begin{lemma}\label{lem:coeff_to_func}
Let \(\widehat{G}_M(\xx)\) and \(\widetilde{G}_M(\xx)\) be two functions of the form \eqref{eq:uN}, with corresponding coefficient vectors \(\widehat{U}=(\hat{u}^{i}_{nj})^\top\) and \(\widetilde{U}=(\tilde{u}^{i}_{nj})^\top\), respectively. Then there exists a constant \(C_1>0\) such that
\[
\|\widehat{G}_M(\xx)-\widetilde{G}_M(\xx)\|_{L^2}
\le
C_1\|\widehat{U}-\widetilde{U}\|_2.
\]
\end{lemma}

\begin{proof}
For the $i$-th solution component, define the linear map $\Phi^i$ that sends a coefficient vector to the corresponding RFM function via \eqref{eq:u_rfm}. Then $\widehat{G}_M - \widetilde{G}_M = \Phi(\widehat{U} - \widetilde{U})$, where $\Phi$ acts componentwise. Since $\Phi$ is a linear map from $\mathbb{R}^{d_u M J_n}$ to $L^2(\Omega)^{d_u}$, its operator norm
\[
C_1 = \sup_{\|v\|_2=1} \|\Phi v\|_{L^2}
\]
is finite (every linear map on a finite-dimensional space is bounded). The claimed inequality follows.
\end{proof}

\begin{remark}\label{rem:C1_bound}
A computable upper bound for $C_1$ can be obtained as follows. Define $\varphi^{i}_{nj}(\xx) := \psi(\widetilde{\xx}^{(n)})\phi_j^{(n)}(\xx)$. When the PoU functions have disjoint subdomain supports (as is the case for $\psi_{\mathrm{a}}$), the cross-subdomain terms vanish. For each subdomain $\Omega_n$, let $\Phi_n$ be the Gram matrix with entries $(\Phi_n)_{jk} = \int_{\Omega_n} \varphi^{i}_{nj}(\xx)\,\varphi^{i}_{nk}(\xx)\,d\xx$. Then $C_1^2 \le \max_{n,i} \lambda_{\max}(\Phi_n)$, where $\lambda_{\max}$ denotes the largest eigenvalue. If, in addition, the random features are approximately orthogonal on each subdomain, $C_1$ reduces to $\max_{n,j,i}\|\varphi^{i}_{nj}\|_{L^2}$.
\end{remark}

The following theorem combines the coefficient error, the stage-wise approximation error, and the IMEX time-discretization error.

\begin{theorem}\label{thm:imex_main}
Let \(\uu^{e}(t_n)\) be the exact solution of \eqref{eq:3.1} at time \(t_n\), and let \(\uu^{n}\) be the solution generated by the four-stage IMEX-RK scheme \eqref{eq:3.8a}--\eqref{eq:3.8e}, written in the increment form
\begin{equation}
\uu^{n+1}
=
\uu^{n}
+
\Delta t\,H_{\mathrm{IMEX}}(\xx,t_n,\uu^{n}(\xx),\Delta t).
\label{eq:imex_increment_form}
\end{equation}
Let \(\widetilde{\uu}^{n}\) and \(\widehat{\uu}^{n}\) be the optimal and numerical RFM solutions obtained through the stage representation \eqref{eq:3.9} and the least-squares systems \eqref{eq:3.10}--\eqref{eq:3.11}, respectively. Assume that \(\widehat{\uu}^{0}=\uu^{e}(t_0)\).

Assume that the increment map \(H_{\mathrm{IMEX}}\) is Lipschitz continuous with respect to the numerical state, that is, there exists a constant \(L>0\) such that
\begin{equation}
\|H_{\mathrm{IMEX}}(\xx,t_n,\uu^{n}(\xx),\Delta t)-H_{\mathrm{IMEX}}(\xx,t_n,\uu^{e}(t_n,\xx),\Delta t)\|_{L^2}
\le
L\|\uu^{n}(\xx)-\uu^{e}(t_n,\xx)\|_{L^2},
\label{eq:imex_lipschitz_checked}
\end{equation}
for all \(n=0,\dots,K\).

Next, assume that the optimal RFM stage approximation satisfies
\begin{equation}
\|\widetilde{\uu}^{n}(\xx)-\uu^{n}(\xx)\|_{L^2}
\le
\sqrt{d_u}\,\epsilon,
\qquad n=0,\dots,K,
\label{eq:imex_rfm_assumption_checked}
\end{equation}
for a given \(\epsilon>0\). This bound follows from the universal approximation property in Lemma~\ref{lem:universal}, where $\epsilon$ depends on $N = MJ_n$ and the regularity of $\uu^n$. Finally, assume that the IMEX tableau used in \eqref{eq:3.8a}--\eqref{eq:3.8e} has temporal order \(p_t\), so that the local truncation error satisfies
\begin{equation}
\|\uu^{e}(t_{n+1},\xx)-\uu^{e}(t_n,\xx)-\Delta t\,H_{\mathrm{IMEX}}(\xx,t_n,\uu^{e}(t_n,\xx),\Delta t)\|_{L^2}
\le
C\Delta t^{p_t+1}.
\label{eq:imex_local_error_checked}
\end{equation}

Then there exist constants \(C_2,C_3>0\) such that
\begin{equation}
E_{\delta u}\left[\|\widehat{\uu}^{n}(\xx)-\uu^{e}(t_n,\xx)\|_{L^2}\right]
\le
C_2\sqrt{d_uMJ_n}\,\delta
+
\sqrt{d_u}\,\epsilon
+
C_3\Delta t^{p_t},
\qquad n=0,\dots,K.
\label{eq:imex_main_error_checked}
\end{equation}
\end{theorem}

\begin{proof}
We split the total error into three parts:
\begin{equation}
\begin{aligned}
E_{\delta u}\left[\|\widehat{\uu}^{n}-\uu^{e}(t_n)\|_{L^2}\right]
&\le
E_{\delta u}\left[\|\widehat{\uu}^{n}-\widetilde{\uu}^{n}\|_{L^2}\right]
+
\|\widetilde{\uu}^{n}-\uu^{n}\|_{L^2}
+
\|\uu^{n}-\uu^{e}(t_n)\|_{L^2},
\end{aligned}
\label{eq:imex_split_checked}
\end{equation}
corresponding to the coefficient perturbation error, the stage-wise RFM approximation error, and the pure time-discretization error, respectively.

The first term is controlled by Lemmas~\ref{lem:coeff_error} and~\ref{lem:coeff_to_func}, giving
\begin{equation}
E_{\delta u}\left[\|\widehat{\uu}^{n}-\widetilde{\uu}^{n}\|_{L^2}\right]
\le
C_2\sqrt{d_uMJ_n}\,\delta
\label{eq:imex_part1_checked}
\end{equation}
for some constant \(C_2>0\). The second term is bounded directly by \eqref{eq:imex_rfm_assumption_checked}:
\begin{equation}
\|\widetilde{\uu}^{n}-\uu^{n}\|_{L^2}
\le
\sqrt{d_u}\,\epsilon.
\label{eq:imex_part2_checked}
\end{equation}

It remains to estimate the pure time-discretization error \(\|\uu^{n}-\uu^{e}(t_n)\|_{L^2}\). By \eqref{eq:imex_increment_form},
\[
\uu^{n+1}
=
\uu^{n}
+
\Delta t\,H_{\mathrm{IMEX}}(\xx,t_n,\uu^{n}(\xx),\Delta t).
\]
Hence,
\begin{align}
\|\uu^{n+1}(\xx)-\uu^{e}(t_{n+1},\xx)\|_{L^2}
&=
\Bigl\|
\uu^{n}(\xx)
+\Delta t\,H_{\mathrm{IMEX}}(\xx,t_n,\uu^{n}(\xx),\Delta t)
-\uu^{e}(t_{n+1},\xx)
\Bigr\|_{L^2}
\nonumber\\
&\le
\|\uu^{n}(\xx)-\uu^{e}(t_n,\xx)\|_{L^2}
\nonumber\\
&\quad
+
\Delta t\,
\|H_{\mathrm{IMEX}}(\xx,t_n,\uu^{n}(\xx),\Delta t)
-
H_{\mathrm{IMEX}}(\xx,t_n,\uu^{e}(t_n,\xx),\Delta t)\|_{L^2}
\nonumber\\
&\quad
+
\|\uu^{e}(t_{n+1},\xx)-\uu^{e}(t_n,\xx)-\Delta t\,H_{\mathrm{IMEX}}(\xx,t_n,\uu^{e}(t_n,\xx),\Delta t)\|_{L^2}
\nonumber\\
&\le
(1+\Delta t\,L)\|\uu^{n}(\xx)-\uu^{e}(t_n,\xx)\|_{L^2}
+
C\Delta t^{p_t+1},
\label{eq:imex_recursion_checked}
\end{align}
where the second inequality uses \eqref{eq:imex_lipschitz_checked} and \eqref{eq:imex_local_error_checked}. Applying the discrete Gronwall inequality to \eqref{eq:imex_recursion_checked} with $\uu^0 = \uu^e(t_0)$ yields
\begin{equation}
\|\uu^{n}(\xx)-\uu^{e}(t_n,\xx)\|_{L^2}
\le
C_3\Delta t^{p_t},
\qquad n=0,\dots,K,
\label{eq:imex_part3_checked}
\end{equation}
for some constant \(C_3>0\) independent of \(\Delta t\).

Combining \eqref{eq:imex_split_checked}, \eqref{eq:imex_part1_checked}, \eqref{eq:imex_part2_checked}, and \eqref{eq:imex_part3_checked} gives
\[
E_{\delta u}\left[\|\widehat{\uu}^{n}(\xx)-\uu^{e}(t_n,\xx)\|_{L^2}\right]
\le
C_2\sqrt{d_uMJ_n}\,\delta
+
\sqrt{d_u}\,\epsilon
+
C_3\Delta t^{p_t}.
\]
This completes the proof.
\end{proof}

\begin{remark}\label{rem:error_structure}
Theorem~\ref{thm:imex_main} decomposes the global error into three additive contributions. The first term, $C_2\sqrt{d_uMJ_n}\,\delta$, reflects the sensitivity of the computed solution to perturbations in the least-squares coefficients and grows with the number of degrees of freedom. The second term, $\sqrt{d_u}\,\epsilon$, measures the best-case spatial approximation error of the RFM trial space and decreases as $N = MJ_n$ increases (by Lemma~\ref{lem:universal}). These two terms exhibit opposite dependences on $N$: increasing the number of features improves the approximation capacity but also amplifies the coefficient perturbation. In practice, the singular-value cutoff $\tau_s$ in \eqref{eq:3.11} balances these competing effects. The third term, $C_3\Delta t^{p_t}$, is the standard temporal discretization error controlled by the step size. We note that the bound~\eqref{eq:imex_main_error_checked} treats the three error sources as decoupled at each time level; a fully coupled analysis that tracks the propagation of spatial errors through the time integrator would require additional stability assumptions on the IMEX-RK scheme composed with the RFM projection and is left for future work.
We also note that the Lipschitz assumption~\eqref{eq:imex_lipschitz_checked} is satisfied for the test problems considered in this paper. For polynomial nonlinearities such as $\mathcal{G}(u) = u - u^3$ (Allen--Cahn) or $\mathcal{G}(u) = -\nu\partial_x u$ (Burgers), the increment map $H_{\mathrm{IMEX}}$ is locally Lipschitz in $\uu$ whenever the numerical solution remains bounded, which is the case in all our experiments.
\end{remark}

\section{Numerical Experiments}
\label{sec:numer}

This section reports numerical tests for the IMEX-RFM method on five nonlinear evolution equations. The experiments are organized by increasing difficulty: we begin with a Type-I problem (Allen--Cahn), proceed to two Type-II problems with second-order (Burgers) and third-order (KdV) spatial operators, then test a fourth-order Type-II problem (Cahn--Hilliard), and conclude with a two-dimensional Type-I problem (2D Allen--Cahn). For each test case, we examine temporal convergence by measuring the relative \(L^2\)-error at the final time over a range of step sizes. When an exact solution is unavailable, the reference solution is computed with a sufficiently small time step. The parameter settings for all test cases are summarized in Appendix Table~\ref{tab:parameter_settings_all}.

For comparison, we also test an IMEX-PINN variant that replaces the RFM spatial discretization with a standard PINN while retaining the same IMEX-RK(4,3) time integrator. The PINN uses three hidden layers of width 50 with \(\tanh\) activation. Each IMEX-RK stage is trained independently for 20 epochs using the Adam optimizer with a learning rate of \(10^{-3}\), which was found sufficient for the training loss to plateau at each stage. Running times for both methods are reported in Appendix Table~\ref{tab:time_compare_all}.

All experiments were carried out on a machine with an AMD Ryzen 9 7945HX CPU and an NVIDIA GeForce RTX 4060 Laptop GPU, using Python 3.8, PyTorch 1.13, and CUDA 11.6.

\subsection{1D Allen--Cahn Equation}

We begin with the one-dimensional Allen--Cahn equation, which provides a Type-I test case (the nonlinear term contains no spatial derivatives). The model problem is
\[
\partial_t u=\varepsilon^2 \partial_{xx}u+5(u-u^3),
\qquad \varepsilon=10^{-2},
\qquad (x,t)\in[-1,1]\times(0,1].
\]
The initial condition is
\[
u(x,0)=x^2\cos(\pi x),
\]
and periodic boundary conditions are imposed together with periodicity of the first derivative:
\[
u(-1,t)=u(1,t),\qquad
\partial_x u(-1,t)=\partial_x u(1,t).
\]

The operator splitting assigns $\mathcal{L}(u) = \varepsilon^2 \partial_{xx}u$ (implicit) and $\mathcal{G}(u) = 5(u - u^3)$ (explicit). Since $\mathcal{G}$ is pointwise in $u$, this example belongs to the Type-I class. The spatial discretization follows the domain decomposition in Section~\ref{sec:sub:static}, with first-order continuity conditions imposed at subdomain interfaces. The relative \(L^2\)-error at the final time is shown in Fig.~\ref{fig:AC_results} for a sequence of time step sizes.

\begin{figure}[htbp]
    \centering
    \begin{minipage}{0.49\textwidth}
        \centering
        \includegraphics[width=\textwidth]{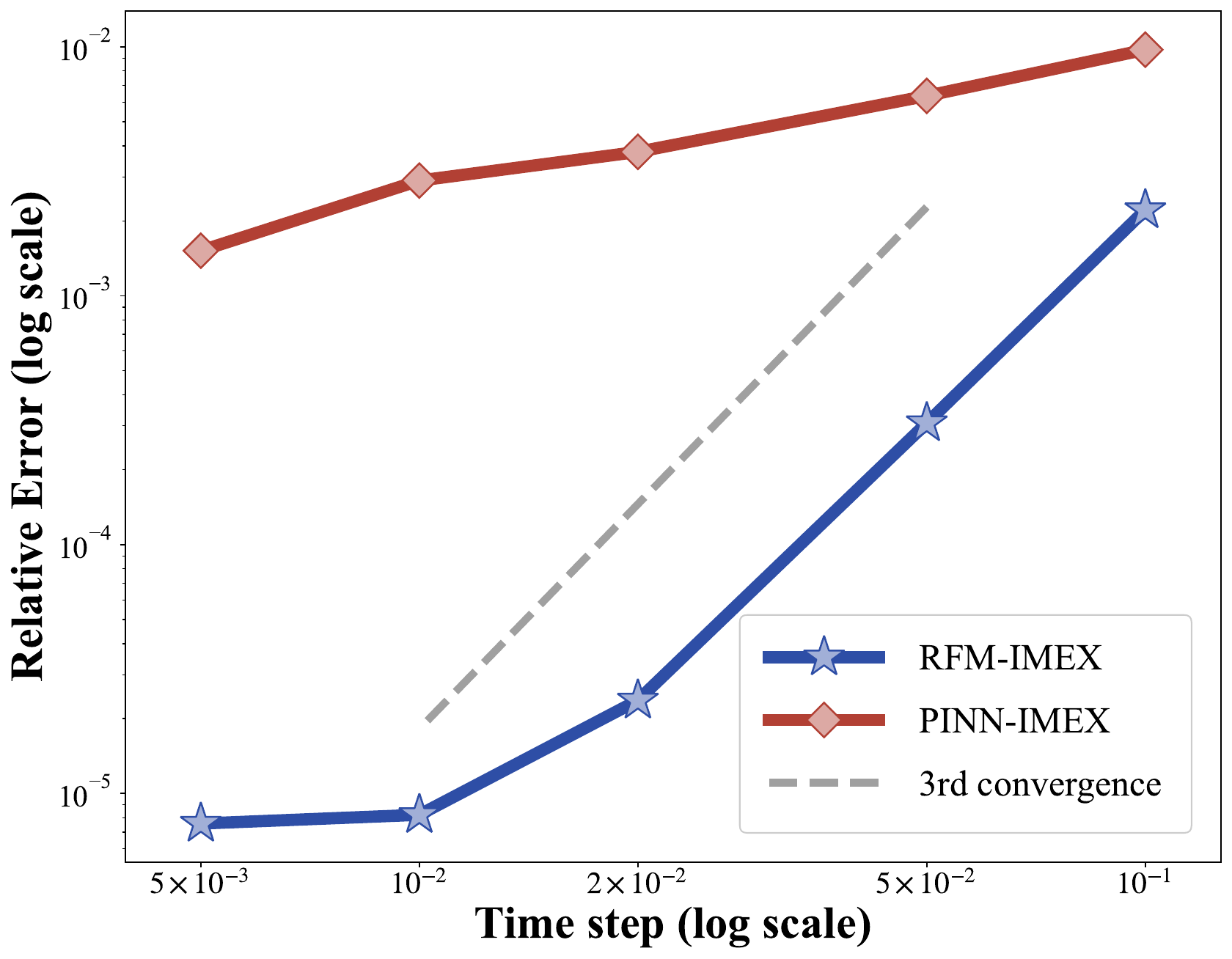}
        \par\vspace{2mm}
        (a) Relative \(L^2\)-error versus time step size
    \end{minipage}\hfill
    \begin{minipage}{0.49\textwidth}
        \centering
        \includegraphics[width=\textwidth]{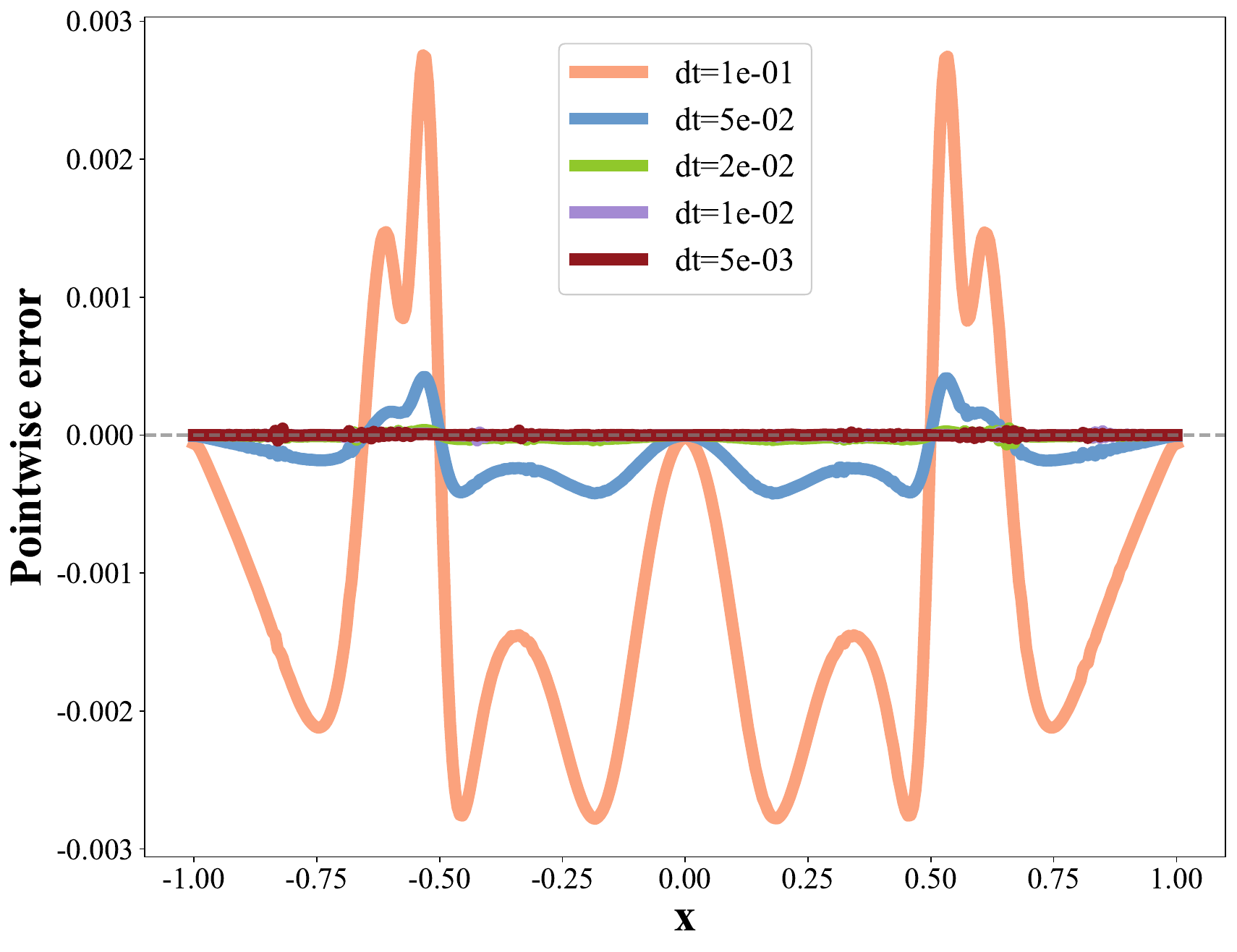}
        \par\vspace{2mm}
        (b) Pointwise errors at \(t=0.5\)
    \end{minipage}
    \caption{1D Allen--Cahn equation: (a) relative \(L^2\)-error at the final time versus time step size; (b) pointwise errors between the reference and predicted solutions at \(t=0.5\).}
    \label{fig:AC_results}
\end{figure}

Figure~\ref{fig:AC_results}(a) shows that the error decreases over most of the tested range, with a slope close to the third-order reference line for $\Delta t > 10^{-2}$. The smallest relative \(L^2\)-error is \(5.97\times10^{-6}\). When the time step is reduced below \(10^{-2}\), the error levels off, indicating that the spatial approximation error (the $\sqrt{d_u}\,\epsilon$ term in Theorem~\ref{thm:imex_main}) becomes dominant over the temporal discretization error ($C_3\Delta t^{p_t}$).

By comparison, the IMEX-PINN method yields a smallest relative \(L^2\)-error of \(1.51\times10^{-3}\), roughly three orders of magnitude larger. Figure~\ref{fig:AC_results}(b) shows that the error at \(t=0.5\) is already small across the whole domain and decreases further as \(\Delta t\) becomes smaller.

\subsection{1D Burgers' Equation}

We next consider the one-dimensional Burgers' equation, which provides a Type-II test case because the nonlinear term contains a spatial derivative. The equation is
\[
\partial_t u + u\partial_x u = \nu \partial_{xx}u,
\qquad \nu=\frac{1}{10\pi},
\qquad (x,t)\in[-1,1]\times(0,1].
\]
The initial condition is
\[
u(x,0)=-\sin(\pi x),
\]
and the boundary conditions are
\[
u(t,-1)=u(t,1)=0,\qquad
\partial_x u(t,-1)=\partial_x u(t,1).
\]

Here $\mathcal{L}(u) = \nu\partial_{xx}u$ and $\mathcal{G}(u) = -u\partial_x u$. Since $\mathcal{G}$ involves $\partial_x u$, this is a Type-II problem: at each IMEX-RK stage, the derivative $\partial_x U_j$ is computed from the RFM representation~\eqref{eq:3.3c} by automatic differentiation. The spatial discretization follows the domain decomposition of Section~\ref{sec:sub:static}, with first-order continuity conditions enforced across subdomain interfaces. The results are shown in Fig.~\ref{fig:burgers_results}.

\begin{figure}[htbp]
    \centering
    \begin{minipage}{0.49\textwidth}
        \centering
        \includegraphics[width=\textwidth]{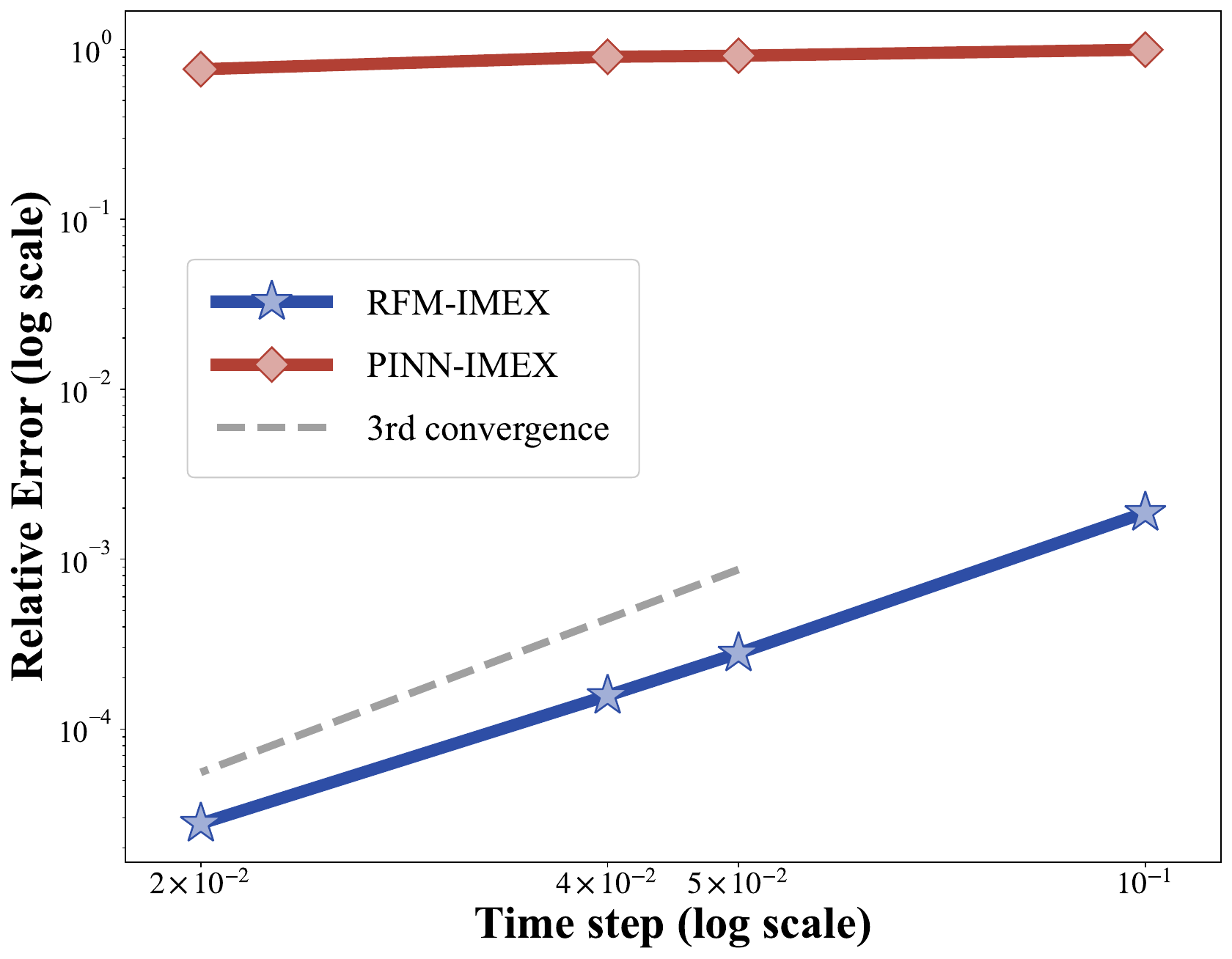}
        \par\vspace{2mm}
        (a) Relative \(L^2\)-error versus time step size
    \end{minipage}\hfill
    \begin{minipage}{0.49\textwidth}
        \centering
        \includegraphics[width=\textwidth]{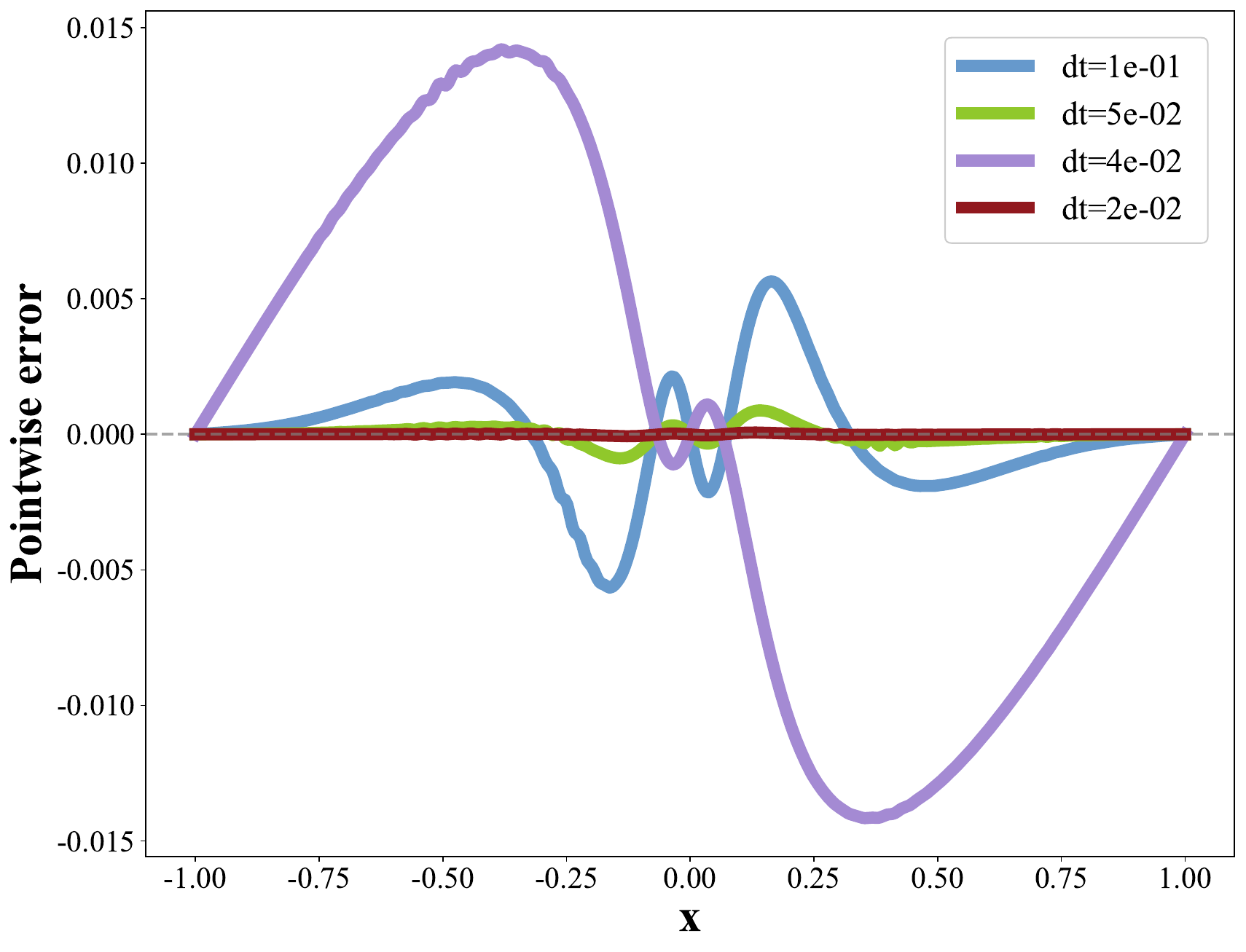}
        \par\vspace{2mm}
        (b) Pointwise errors at \(t=0.5\)
    \end{minipage}
    \caption{1D Burgers' equation: (a) relative \(L^2\)-error at the final time versus time step size; (b) pointwise errors between the reference and predicted solutions at \(t=0.5\).}
    \label{fig:burgers_results}
\end{figure}

Figure~\ref{fig:burgers_results}(a) shows an overall decrease of the error as the time step is refined. The smallest relative \(L^2\)-error is \(2.3\times10^{-5}\), attained at \(\Delta t=2\times10^{-2}\). The convergence curve is not monotone: the non-monotonicity is attributed to the variance of the random feature sampling, which introduces an $O(\delta)$ stochastic perturbation at each time level (cf.\ the first term in~\eqref{eq:imex_main_error_checked}). Despite these fluctuations, the overall trend is consistent with the third-order reference slope. By contrast, the IMEX-PINN error grows rapidly as time advances, preventing the numerical solution from tracking the reference at later times. Figure~\ref{fig:burgers_results}(b) shows that the error at \(t=0.5\) remains small for most choices of \(\Delta t\), which indicates good accuracy. The relatively larger error for \(\Delta t=4\times10^{-2}\) is explained by the fact that \(t=0.5\) is not an integer multiple of the time step, and the plotted curve therefore corresponds to the solution at the preceding time level.

\subsection{1D Korteweg--De Vries Equation}

The next example is the one-dimensional Korteweg--De Vries (KdV) equation, a dispersive Type-II problem with a third-order spatial operator. We consider
\[
\partial_t u + u \partial_x u + \alpha^2 \partial_{xxx} u = 0,
\qquad \alpha=0.022,
\qquad (x,t)\in[-1,1]\times(0,1].
\]
The initial condition is
\[
u(x,0)=\cos(\pi x),
\]
and periodic boundary conditions are imposed:
\[
u(-1,t)=u(1,t),\qquad
\partial_x u(-1,t)=\partial_x u(1,t),\qquad
\partial_{xx} u(-1,t)=\partial_{xx} u(1,t).
\]

The splitting is $\mathcal{L}(u) = \alpha^2\partial_{xxx}u$ and $\mathcal{G}(u) = -u\partial_x u$, again a Type-II problem. Since the linear part involves a third-order operator, we impose second-order continuity conditions at subdomain interfaces to ensure sufficient smoothness of the trial functions. The final-time relative \(L^2\)-error for a range of time step sizes is shown in Fig.~\ref{fig:kdv_results}.

\begin{figure}[htbp]
    \centering
    \begin{minipage}{0.49\textwidth}
        \centering
        \includegraphics[width=\textwidth]{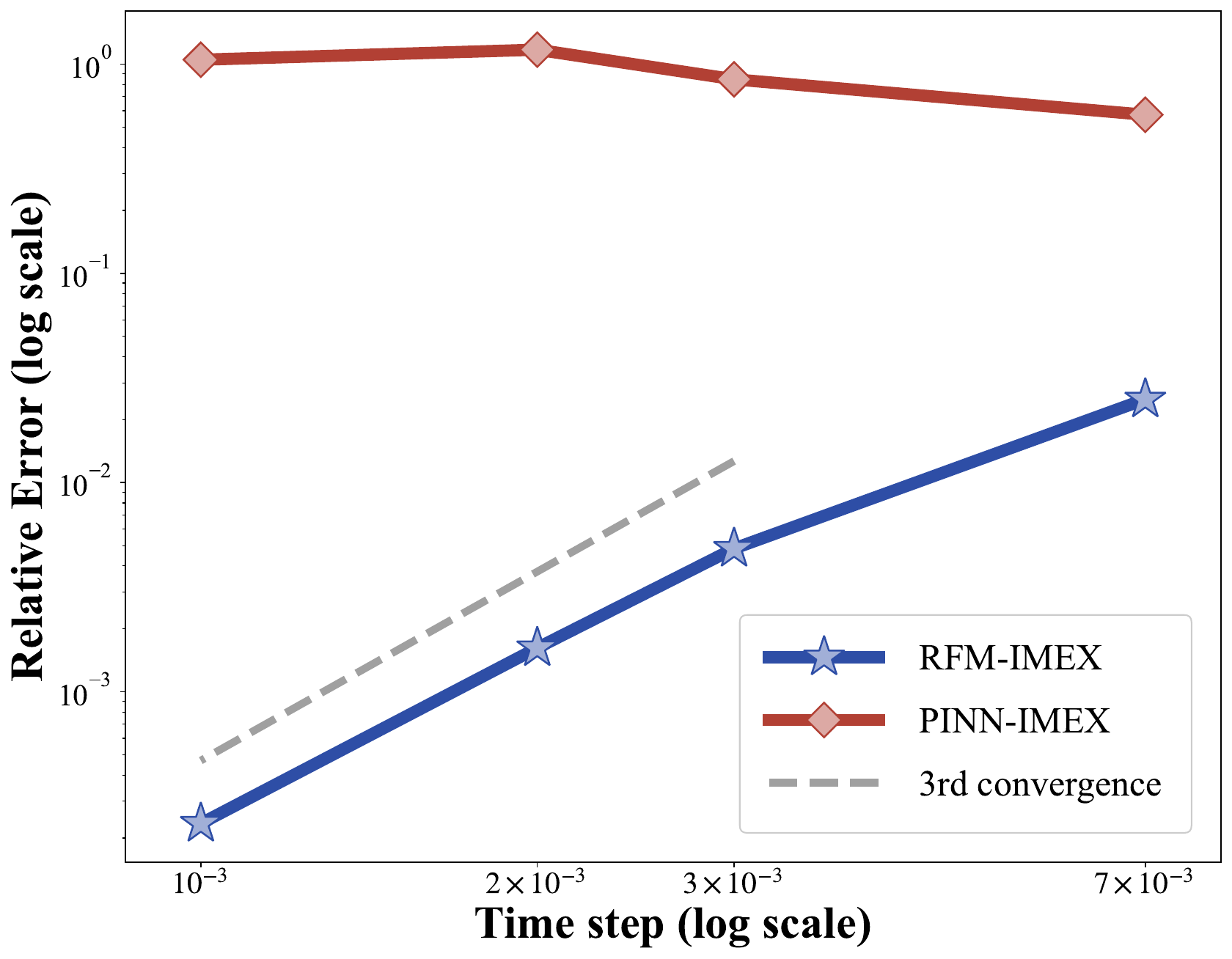}
        \par\vspace{2mm}
        (a) Relative \(L^2\)-error versus time step size
    \end{minipage}\hfill
    \begin{minipage}{0.49\textwidth}
        \centering
        \includegraphics[width=\textwidth]{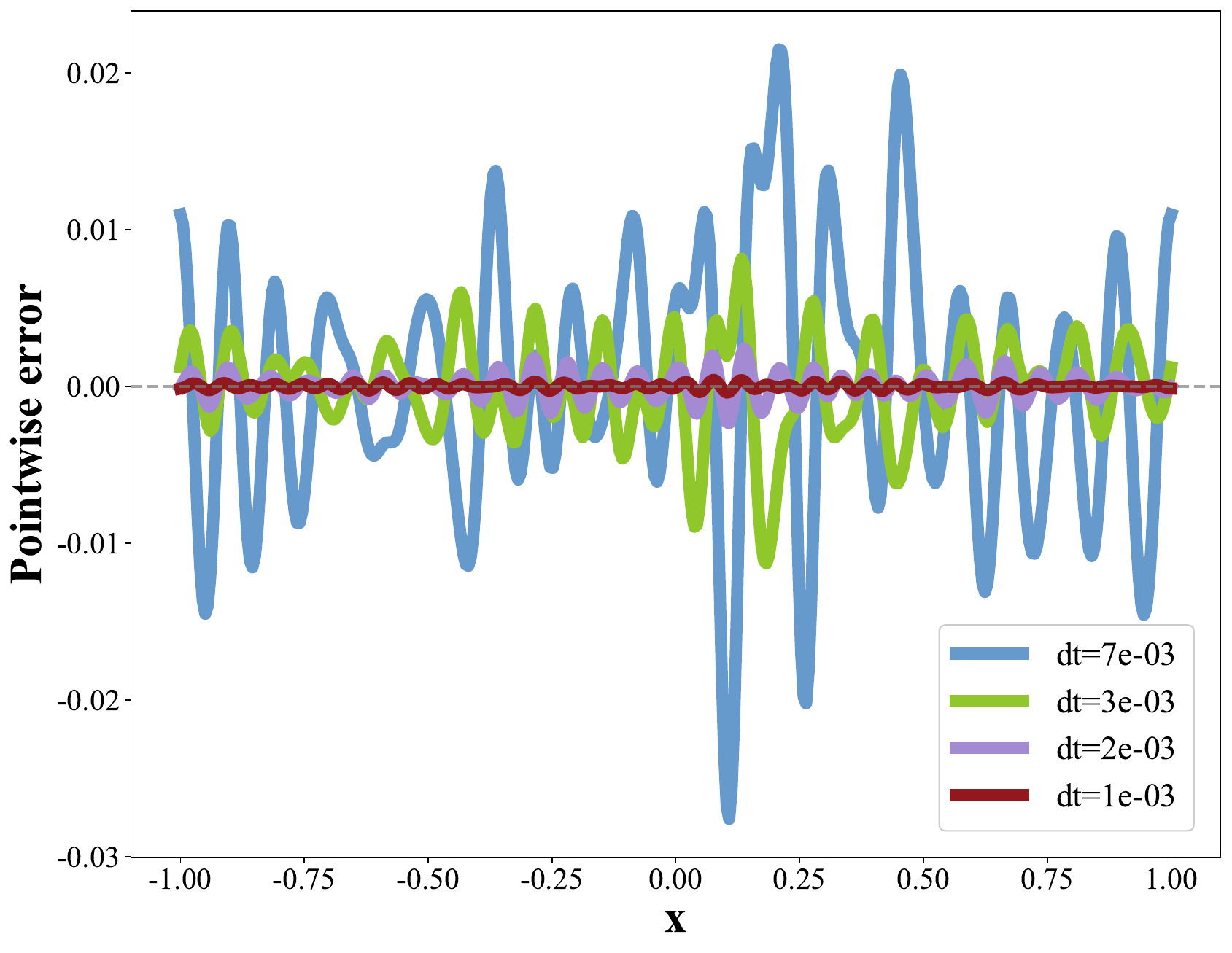}
        \par\vspace{2mm}
        (b) Pointwise errors at \(t=0.5\)
    \end{minipage}
    \caption{1D Korteweg--De Vries equation: (a) relative \(L^2\)-error at the final time versus time step size; (b) pointwise errors between the reference and predicted solutions at \(t=0.5\).}
    \label{fig:kdv_results}
\end{figure}

Figure~\ref{fig:kdv_results}(a) shows an overall decrease of the error with decreasing step size. The smallest relative \(L^2\)-error is \(2.43\times10^{-4}\), larger than the Allen--Cahn and Burgers results. This is consistent with the higher spatial regularity requirements imposed by the third-order operator. Over most of the tested range, the convergence rate agrees with the third-order reference slope. As in the previous tests, the IMEX-PINN method yields substantially larger errors. Figure~\ref{fig:kdv_results}(b) shows the corresponding pointwise errors of reference and predicted solutions at \(t=0.5\).

\subsection{1D Cahn--Hilliard Equation}

We next consider the one-dimensional Cahn--Hilliard equation, a fourth-order nonlinear problem that provides the most demanding Type-II test.

We follow the formulation in~\cite{ijcai2024p573}:
\[
\partial_t u-\partial_{xx}\big(\gamma_1(u^3-u)-\gamma_2\partial_{xx}u\big)=0,
\qquad (x,t)\in[-1,1]\times(0,1].
\]
The parameters are chosen as \(\gamma_1=0.01\) and \(\gamma_2=10^{-6}\). Periodic boundary conditions are imposed up to the third derivative:
\begin{align*}
u(-1,t)&=u(1,t), &\quad \partial_x u(-1,t)&=\partial_x u(1,t),\\
\partial_{xx}u(-1,t)&=\partial_{xx}u(1,t), &\quad \partial_{xxx}u(-1,t)&=\partial_{xxx}u(1,t).
\end{align*}
The initial condition is
\[
u(x,0)=\sum_{i=1}^{2}A_i\sin(k_i x+\phi_i),
\]
where the wave numbers are given by \(k_i=2\pi n_i\), with \(n_i\in\{1,2,\dots,8\}\) for \(i=1,2\). The amplitudes \(A_i\in[0,1]\) and phases \(\phi_i\in(0,2\pi)\) are sampled independently from a uniform distribution. A single fixed realization is used for all reported experiments.

The splitting assigns $\mathcal{L}(u) = -\gamma_2\partial_{xxxx}u$ (implicit) and $\mathcal{G}(u) = \partial_{xx}[\gamma_1(u^3 - u)]$ (explicit). The explicit term involves second derivatives of a nonlinear function of $u$, making this a Type-II problem with both strong nonlinearity and high spatial order. Following Section~\ref{sec:sub:static}, we impose third-order continuity conditions at subdomain interfaces. The relative \(L^2\)-error at the final time is shown in Fig.~\ref{fig:CH_results}.

\begin{figure}[htbp]
    \centering
    \begin{minipage}{0.49\textwidth}
        \centering
        \includegraphics[width=\textwidth]{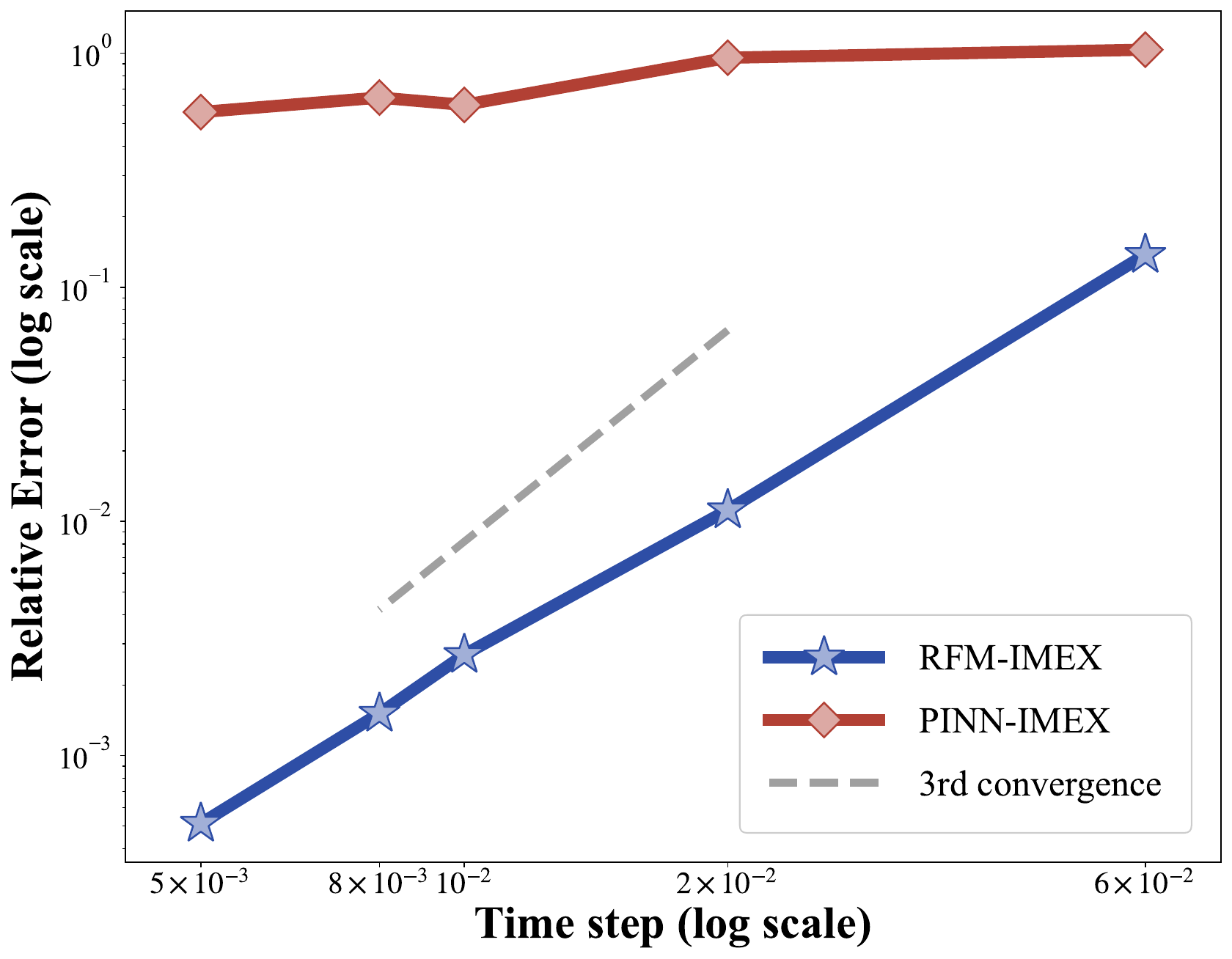}
        \par\vspace{2mm}
        (a) Relative \(L^2\)-error versus time step size
    \end{minipage}\hfill
    \begin{minipage}{0.49\textwidth}
        \centering
        \includegraphics[width=\textwidth]{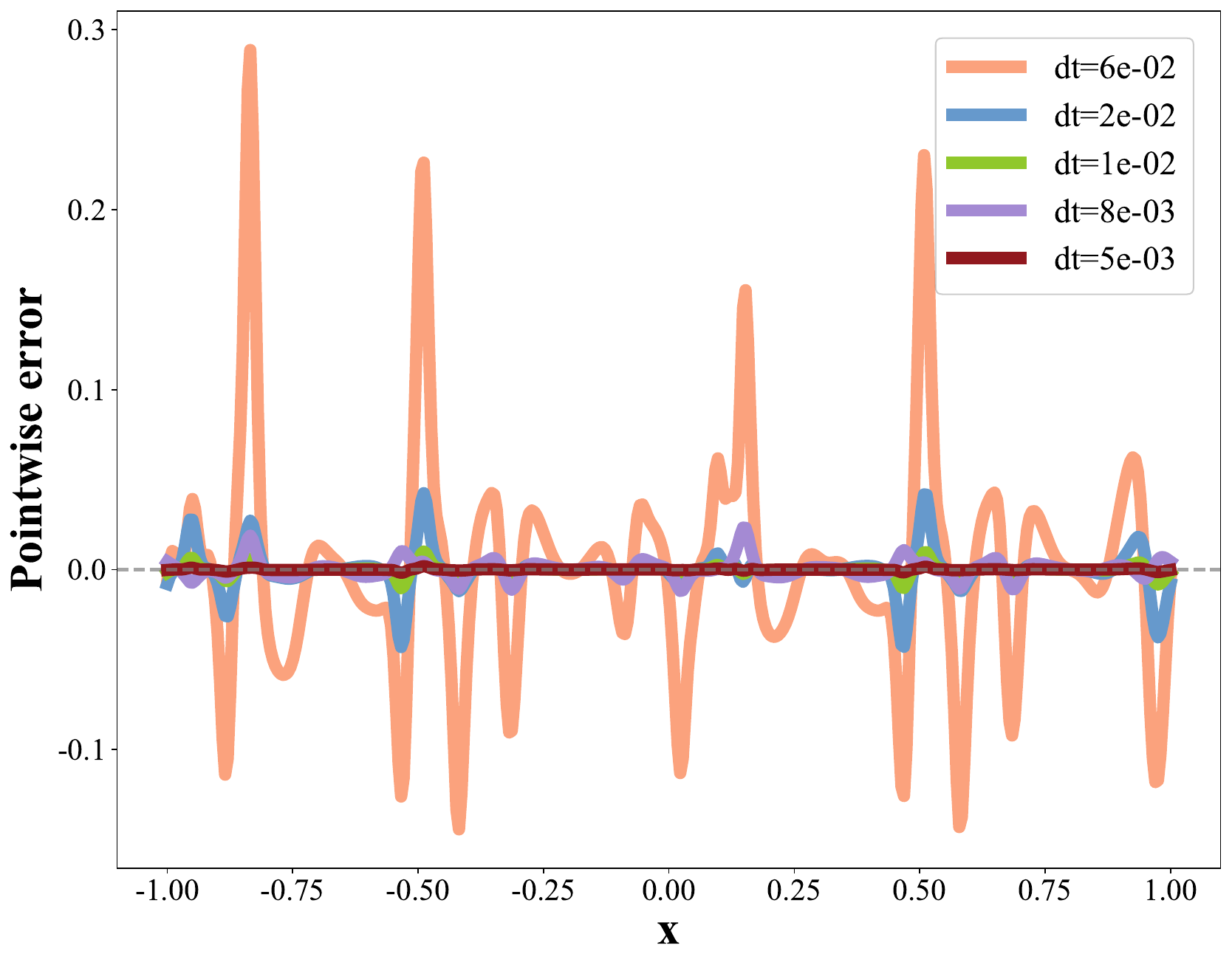}
        \par\vspace{2mm}
        (b) Pointwise errors at \(t=0.5\)
    \end{minipage}
    \caption{1D Cahn--Hilliard equation: (a) relative \(L^2\)-error at the final time versus time step size; (b) pointwise errors between the reference and predicted solutions at \(t=0.5\).}
    \label{fig:CH_results}
\end{figure}

Figure~\ref{fig:CH_results}(a) shows that the smallest relative \(L^2\)-error is \(5.11\times10^{-4}\), the largest among the four one-dimensional tests. This is expected: the fourth-order spatial operator amplifies both the RFM approximation error and the sensitivity of the least-squares system to the random feature sampling. As in the Burgers and KdV tests, the convergence curve exhibits mild fluctuations due to random feature variance, though the overall trend is consistent with third-order convergence. The IMEX-RFM method again outperforms the IMEX-PINN variant. Figure~\ref{fig:kdv_results}(b) shows the corresponding pointwise errors of reference and predicted solutions at \(t=0.5\).

\subsection{2D Allen--Cahn Equation}

For the final example, we consider the two-dimensional Allen--Cahn equation to test the method in higher spatial dimensions. The equation is
\[
\partial_t u = \varepsilon^2 \Delta u + \left(u-u^3\right),
\qquad (x,y)\in\Omega,\ t\in(0,T].
\]
The computational domain, terminal time, and parameter are
\[
\Omega=[-1,1]\times[-1,1],
\qquad \varepsilon=10^{-2},
\qquad T=1 .
\]
The initial condition is
\[
u(x,y,0)=0.05\sin(\pi x)\sin(\pi y),
\]
and periodic boundary conditions are imposed in both spatial directions:
\[
u(-1,y,t)=u(1,y,t),\quad u_x(-1,y,t)=u_x(1,y,t),\qquad y\in(-1,1),
\]
\[
u(x,-1,t)=u(x,1,t),\quad u_y(x,-1,t)=u_y(x,1,t),\qquad x\in(-1,1).
\]

As in the one-dimensional case, the nonlinear term is pointwise in $u$ (Type-I). First-order continuity conditions are enforced at the subdomain interfaces. A comparison with the IMEX-PINN method is presented in Figure~\ref{fig:2D_AC_Convergence}. The minimum relative \(L^2\)-error attained by IMEX-RFM is \(1.24\times10^{-3}\).
\begin{figure}
    \centering
    \includegraphics[width=0.5\linewidth]{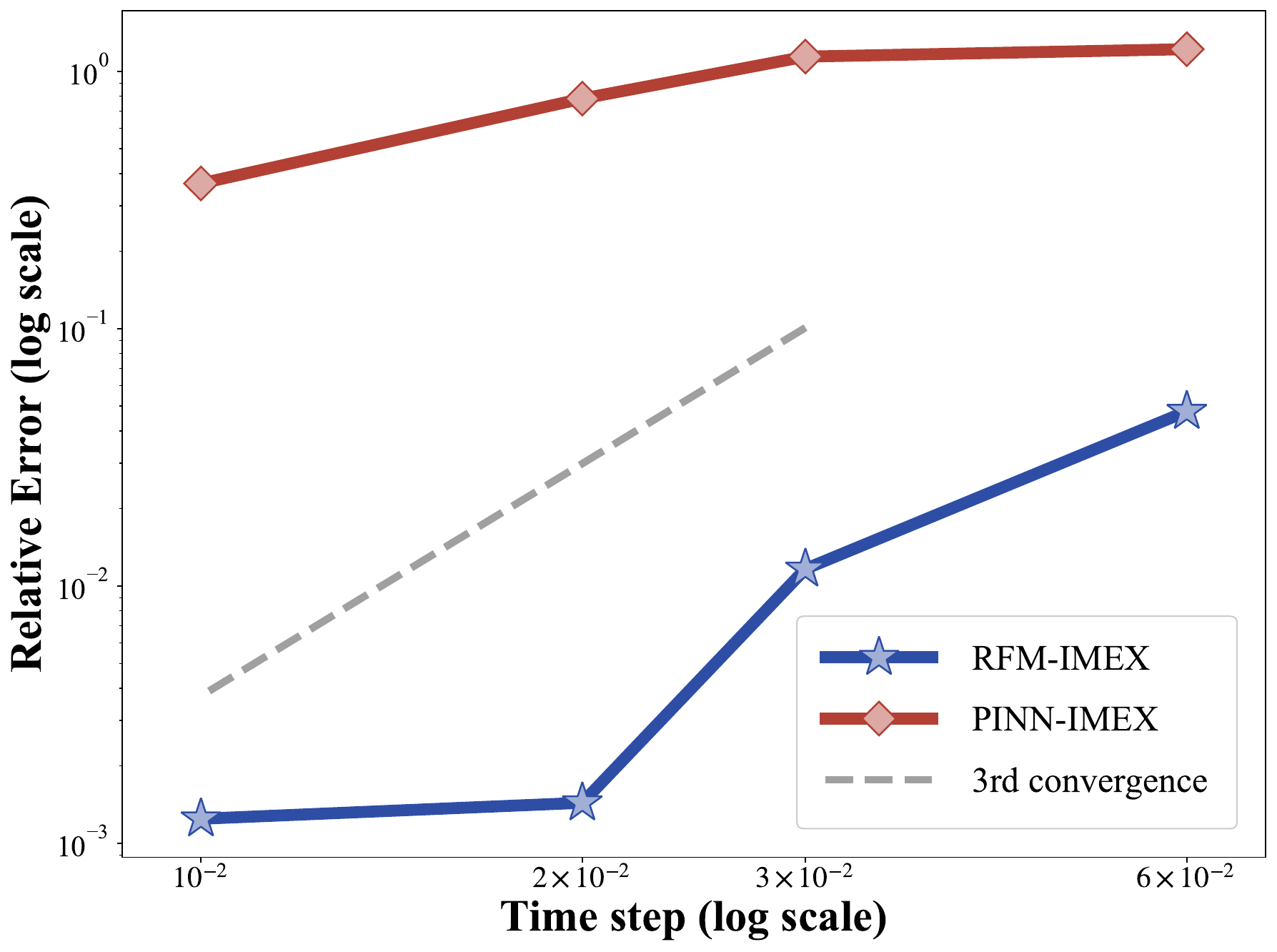}
    \caption{Comparison of the relative \(L^2\)-errors of IMEX-RFM and IMEX-PINN for the 2D Allen--Cahn equation.}
    \label{fig:2D_AC_Convergence}
\end{figure}

Although not all individual step sizes show clear third-order convergence, a comparison of the cases \(\Delta t=6\times10^{-2}\) and \(\Delta t=3\times10^{-2}\) (a factor-of-two refinement yielding roughly an eight-fold error reduction) indicates that the overall temporal convergence behavior is consistent with third-order accuracy. The local discrepancy at other step sizes is attributed to the increased variance of the random feature sampling in two dimensions, where the number of random features per subdomain ($J_n = 200$) is smaller than in the one-dimensional tests. As in the 1D Allen--Cahn test, the error does not continue to decrease for very small $\Delta t$, confirming that the spatial error floor is reached.

We now summarize the results across all five test cases. The method achieves relative $L^2$-errors ranging from $5.97 \times 10^{-6}$ (1D Allen--Cahn) to $5.11 \times 10^{-4}$ (1D Cahn--Hilliard) in one dimension and $1.24 \times 10^{-3}$ in two dimensions. In all cases, the convergence trends are consistent with the third-order IMEX-RK(4,3) scheme. Among the one-dimensional tests, the Cahn--Hilliard equation yields the largest errors, which is expected given the fourth-order spatial operator and the sensitivity of the phase-field dynamics to the spatial approximation. The non-monotone convergence curves in the Burgers and Cahn--Hilliard tests are attributed to the stochastic nature of the random feature initialization (cf.\ the coefficient perturbation term in Theorem~\ref{thm:imex_main}). The running-time comparison in Appendix Table~\ref{tab:time_compare_all} confirms that IMEX-RFM is consistently faster than IMEX-PINN across all test cases.

\section{Conclusion}
\label{sec:conclusion}

We have presented a discrete-time RFM framework for nonlinear time-dependent PDEs combined with an implicit-explicit Runge--Kutta time discretization. The method advances the solution sequentially in time and preserves stage-wise linearity in the RFM coefficients, so each stage reduces to a linear least-squares problem. We derived a global error bound (Theorem~\ref{thm:imex_main}) for the fully discrete method that separates the contributions of coefficient perturbation, stage-wise RFM approximation, and IMEX temporal discretization. Numerical experiments on five test problems (Allen--Cahn, Burgers, KdV, and Cahn--Hilliard in one dimension, Allen--Cahn in two dimensions) confirm third-order temporal convergence and show that the method achieves higher accuracy than an IMEX-PINN variant at a fraction of the computational cost.

We note several limitations of the current work.

\begin{enumerate}
\item \textit{Conditional stability.} The fully discrete method inherits the stability region of the underlying IMEX-RK(4,3) tableau. The present analysis does not provide an unconditional stability guarantee for the RFM discretization, so strongly stiff problems may still require small time step sizes.

\item \textit{Random feature variance.} The spatial approximation depends on the sampled random features, which are fixed after initialization. Different random seeds can lead to noticeably different errors, consistent with the non-monotone convergence curves observed in the Burgers and Cahn--Hilliard tests.

\item \textit{Absence of structure preservation.} The current formulation does not enforce problem-specific invariants such as energy dissipation (Allen--Cahn, Cahn--Hilliard) or mass conservation.

\item \textit{Limited higher-dimensional testing.} Although the domain decomposition and least-squares assembly extend directly to higher dimensions, the growth of the design matrix with spatial dimension has not been systematically studied.
\end{enumerate}

These limitations suggest several directions for future work: (i) systematic experiments in two and three dimensions to assess how accuracy and cost scale with spatial dimension; (ii) adaptive strategies for selecting $M$, $J_n$, $Q$, and $\tau_s$ to reduce manual parameter tuning; (iii) structure-preserving extensions of the least-squares formulation to enforce physical invariants; and (iv) variance analysis over multiple random seeds, possibly combined with ensemble averaging or quasi-random sampling, to quantify and reduce the sensitivity to random initialization.

\section{Appendix}

This appendix collects supplementary material for the numerical experiments in Section~\ref{sec:numer}. Section~6.1 records the coefficients of the IMEX-RK(4,3) scheme. Section~6.2 summarizes the parameter settings for all test equations. Section~6.3 reports the total running times of IMEX-RFM and IMEX-PINN. Section~6.4 presents heatmap comparisons between the reference and predicted solutions, showing the spatial distribution of the pointwise errors.

\subsection{Coefficients of the IMEX-RK(4,3) Scheme}

The implicit and explicit Butcher tableaux of the four-stage IMEX-RK(4,3) scheme used in all experiments are listed in Table~\ref{tab:imex-rk-coeff}; see~\cite{ascher1995implicit} for a derivation.

\begin{table}[htbp]
\centering
\begin{minipage}{0.4\textwidth}
\centering
\begin{tabular}{c|cccc}
\multicolumn{5}{c}{Implicit coefficients $A=(a_{ij})$} \\
\hline
$\frac12$ & $\frac12$ & 0 & 0 & 0 \\
$\frac23$ & $\frac16$ & $\frac12$ & 0 & 0 \\
$\frac12$ & $-\frac12$ & $\frac12$ & $\frac12$ & 0 \\
1 & $\frac32$ & $-\frac32$ & $\frac12$ & $\frac12$ \\
\hline
 & $\frac32$ & $-\frac32$ & $\frac12$ & $\frac12$
\end{tabular}
\end{minipage}
\hspace{0.01\textwidth}
\begin{minipage}{0.4\textwidth}
\centering
\begin{tabular}{c|cccc}
\multicolumn{5}{c}{Explicit coefficients $\hat A=(\hat a_{ij})$} \\
\hline
$\frac12$ & $\frac12$ & 0 & 0 & 0 \\
$\frac23$ & $\frac{11}{18}$ & $\frac1{18}$ & 0 & 0 \\
$\frac12$ & $\frac56$ & $-\frac56$ & $\frac12$ & 0 \\
1 & $\frac14$ & $\frac74$ & $\frac34$ & $-\frac74$ \\
\hline
 & $\frac14$ & $\frac74$ & $\frac34$ & $-\frac74$
\end{tabular}
\end{minipage}
\caption{Butcher tableaux of the four-stage IMEX-RK(4,3) scheme (third-order in time) used in this work~\cite{ascher1995implicit}.}
\label{tab:imex-rk-coeff}
\end{table}

\subsection{Parameter Settings}

Table~\ref{tab:parameter_settings_all} lists the parameter settings for all test equations.All singular-value cutoff \(\tau_s\) are set to $1\mathrm{e}{-16}$ in the computation. For the 2D Allen--Cahn equation, the values of \(R_m\), \(M\), \(J_n\), and \(Q\) apply independently to each spatial direction.

\begin{table}[htbp]
\centering

\setlength{\tabcolsep}{6pt} 
\begin{tabularx}{\linewidth}{l c c c c >{\centering\arraybackslash}X}
\toprule
\textbf{Equation} & $R_m$ & $M$ & $J_n$ & $Q$ & $\Delta t$ set \\
\midrule
1D Allen--Cahn
& 20 & 8 & 500 & 100 & $[1\mathrm{e}{-1},\,  4\mathrm{e}{-2},\, 2\mathrm{e}{-2},\, 1\mathrm{e}{-2},\, 5\mathrm{e}{-3},\, 1\mathrm{e}{-3}]$ \\

1D Burgers'
& 20 & 2 & 400 & 100 & $[1\mathrm{e}{-1},\, 5\mathrm{e}{-2},\, 4\mathrm{e}{-2},\, 2\mathrm{e}{-2},\, 1\mathrm{e}{-3}]$ \\

1D Korteweg--De Vries
& 16 & 8 & 400 & 100 & $[7\mathrm{e}{-3},\,  3\mathrm{e}{-3},\, 2\mathrm{e}{-3},\, 1\mathrm{e}{-3},\, 8\mathrm{e}{-4}]$ \\

1D Cahn--Hilliard
& 14 & 10 & 500 & 100 & $[6\mathrm{e}{-2},\, 2\mathrm{e}{-2},\, 1\mathrm{e}{-2},\, 8\mathrm{e}{-3},\, 5\mathrm{e}{-3},\, 1\mathrm{e}{-3}]$ \\

2D Allen--Cahn
& 1 & 2 & 200 & 25 & $[6\mathrm{e}{-2},\, 3\mathrm{e}{-2},\, 2\mathrm{e}{-2},\, 1\mathrm{e}{-2},\, 1\mathrm{e}{-3}]$ \\
\bottomrule
\end{tabularx}
\caption{Parameter settings for the test equations.}
\label{tab:parameter_settings_all}
\end{table}

\subsection{Total Running Time Comparisons}

Table~\ref{tab:time_compare_all} reports the total wall-clock times (in seconds) of IMEX-RFM and IMEX-PINN for all test equations. In every case, IMEX-RFM is faster, with speedup factors ranging from approximately 2.8$\times$ (1D Allen--Cahn) to 15.1$\times$ (1D Burgers).

\begin{table}[htbp]
\centering
\renewcommand{\arraystretch}{1.2}
\setlength{\tabcolsep}{10pt}
\begin{tabular}{lccc}
\toprule
\textbf{Equation} & \textbf{IMEX-RFM (s)} & \textbf{IMEX-PINN (s)} & \textbf{Speedup} \\
\midrule
1D Allen--Cahn            & 374.5  & 1050.8 & $2.8\times$ \\
1D Burgers'               & 46.0   & 694.4  & $15.1\times$ \\
1D Korteweg--De Vries     & 738.0  & 3905.9 & $5.3\times$ \\
1D Cahn--Hilliard         & 908.6  & 4937.8 & $5.4\times$ \\
2D Allen--Cahn            & 313.7  & 1335.7 & $4.3\times$ \\
\bottomrule
\end{tabular}
\caption{Total running times (in seconds) and speedup factors of IMEX-RFM over IMEX-PINN.}
\label{tab:time_compare_all}
\end{table}

\subsection{Heatmap Comparisons}

Figures~8--17 present space-time heatmap comparisons for all test equations at representative time step sizes. Each row shows the reference solution (left), the predicted solution (center), and the absolute pointwise error (right). These plots complement the convergence curves in Section~\ref{sec:numer} by revealing the spatial distribution of the errors.

\begin{figure}[htbp]
    \centering
    \includegraphics[width=1.0\textwidth]{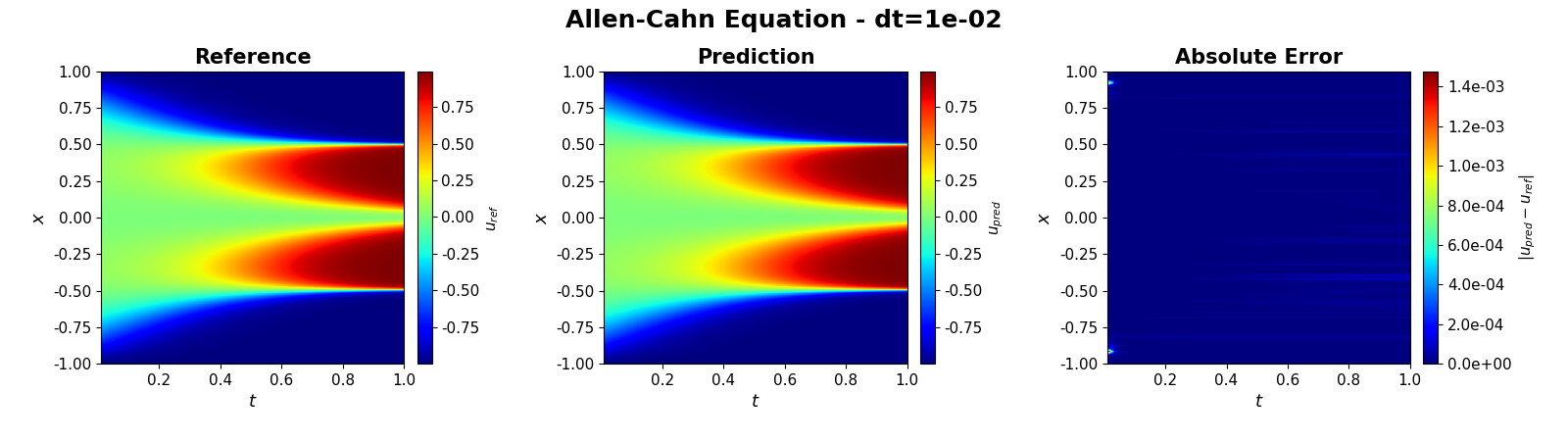}
    \caption{1D Allen--Cahn equation with $\Delta t=10^{-2}$: reference solution (left), predicted solution (center), and absolute error (right).}
\end{figure}

\begin{figure}[htbp]
    \centering
    \includegraphics[width=1.0\textwidth]{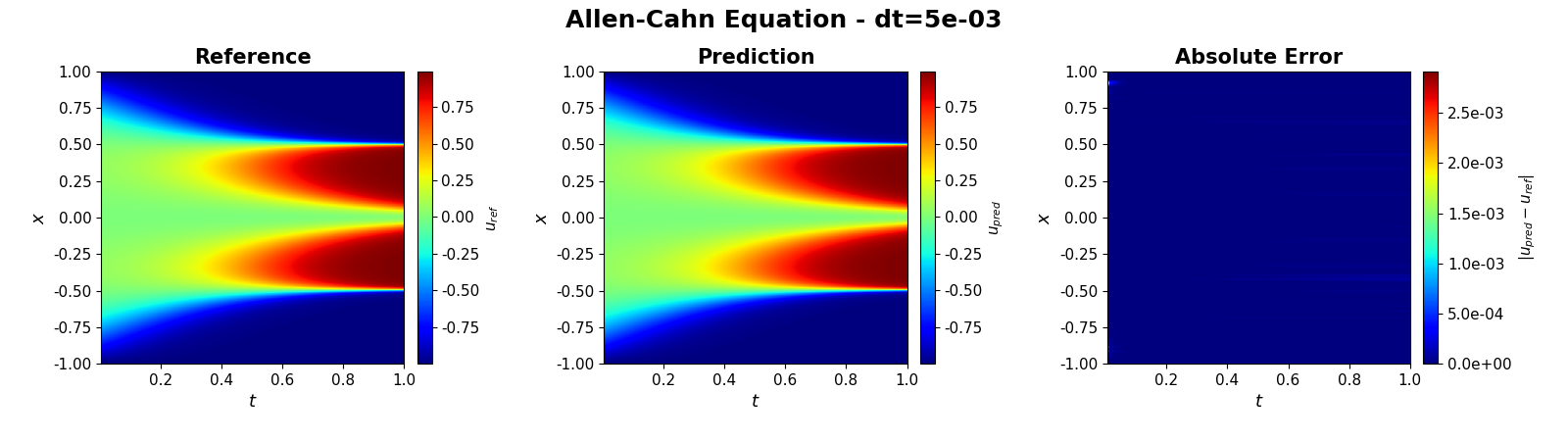}
    \caption{1D Allen--Cahn equation with $\Delta t=5\times10^{-3}$: reference solution (left), predicted solution (center), and absolute error (right).}
\end{figure}

\begin{figure}[htbp]
    \centering
    \includegraphics[width=1.0\textwidth]{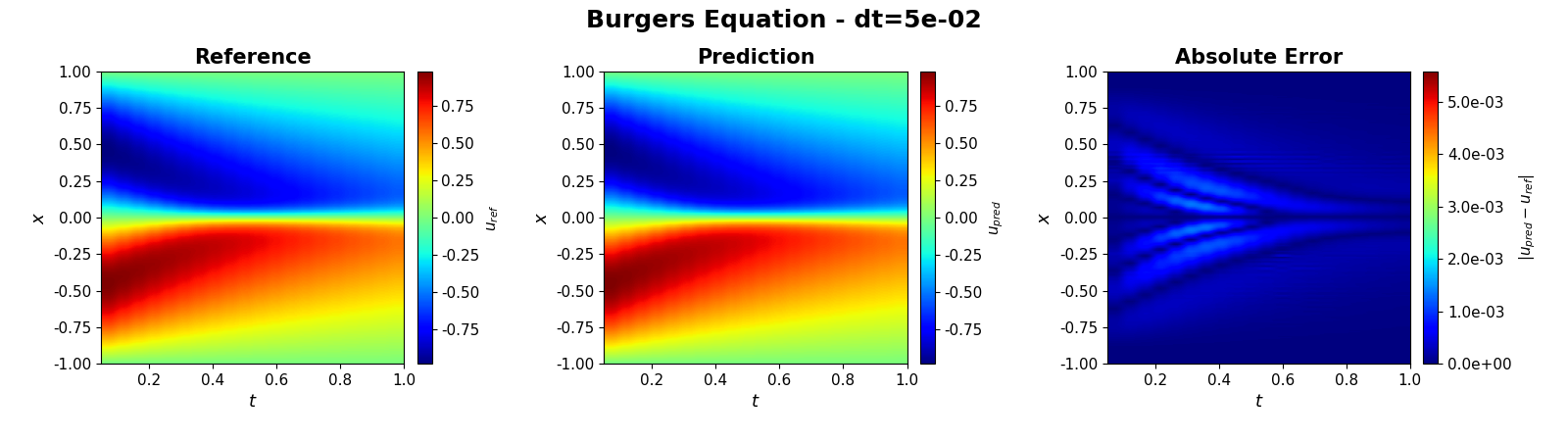}
    \caption{1D Burgers' equation with $\Delta t=5\times10^{-2}$: reference solution (left), predicted solution (center), and absolute error (right).}
\end{figure}

\begin{figure}[htbp]
    \centering
    \includegraphics[width=1.0\textwidth]{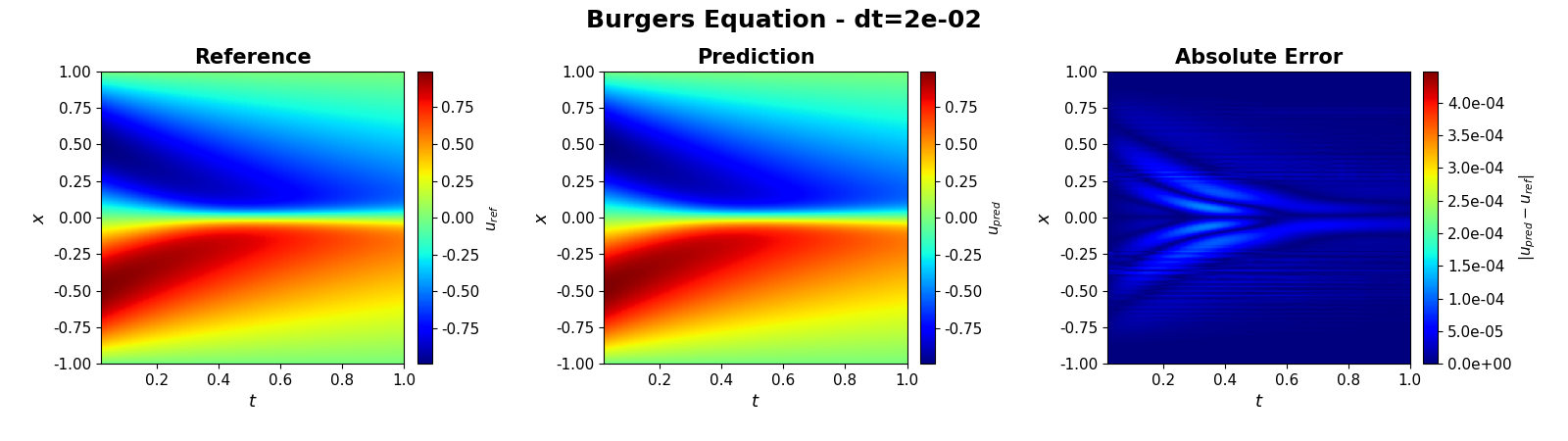}
    \caption{1D Burgers' equation with $\Delta t=2\times10^{-2}$: reference solution (left), predicted solution (center), and absolute error (right).}
\end{figure}

\begin{figure}[htbp]
    \centering
    \includegraphics[width=1.0\textwidth]{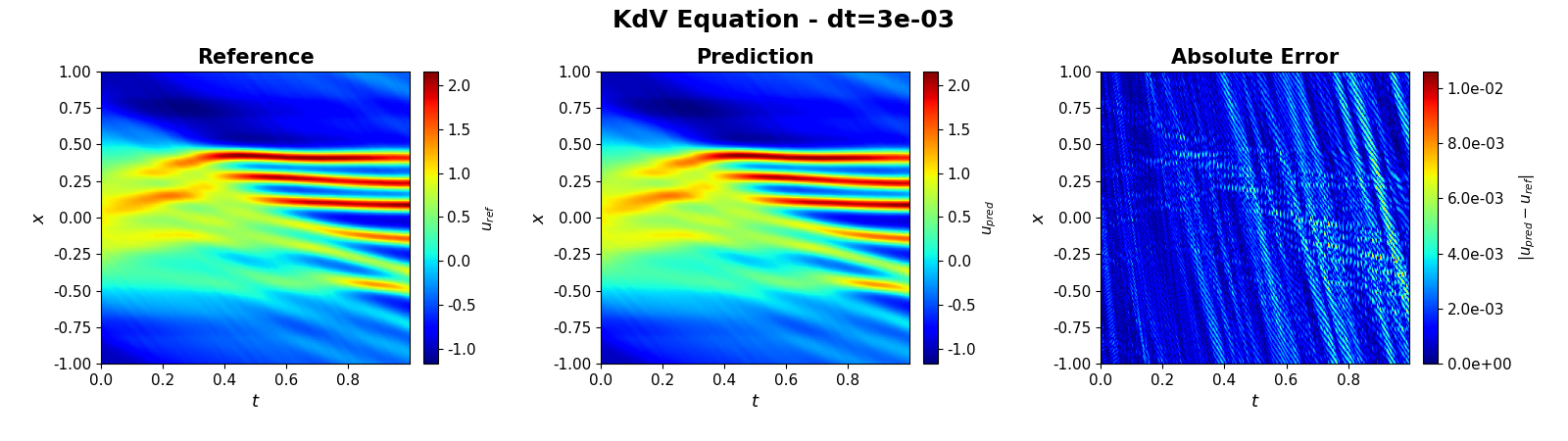}
    \caption{1D KdV equation with $\Delta t=3\times10^{-3}$: reference solution (left), predicted solution (center), and absolute error (right).}
\end{figure}

\begin{figure}[htbp]
    \centering
    \includegraphics[width=1.0\textwidth]{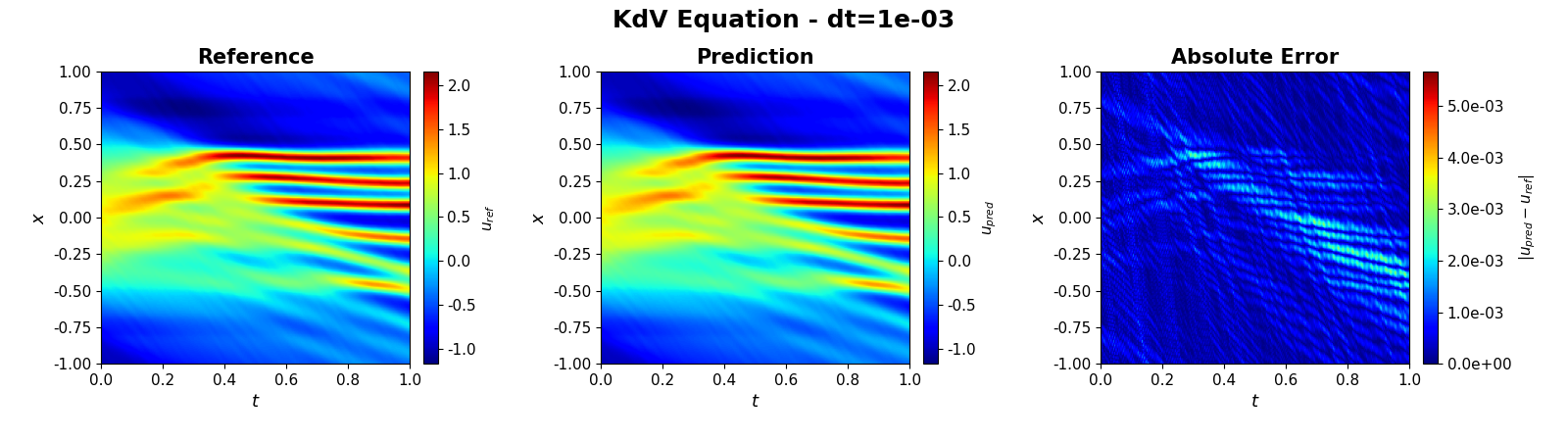}
    \caption{1D KdV equation with $\Delta t=10^{-3}$: reference solution (left), predicted solution (center), and absolute error (right).}
\end{figure}

\begin{figure}[htbp]
    \centering
    \includegraphics[width=1.0\textwidth]{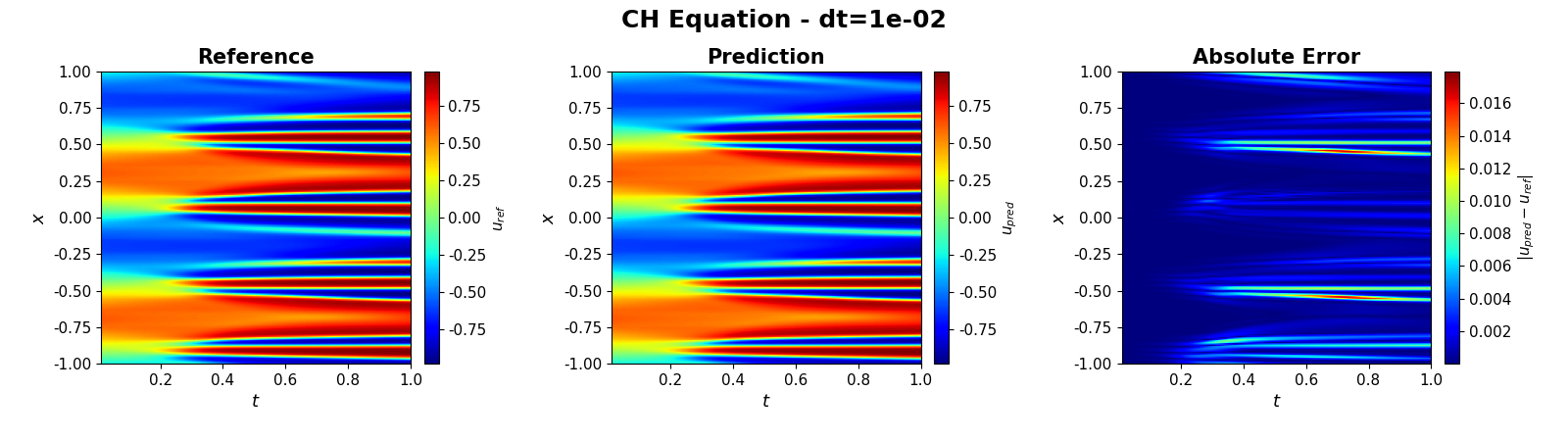}
    \caption{1D Cahn--Hilliard equation with $\Delta t=10^{-2}$: reference solution (left), predicted solution (center), and absolute error (right).}
\end{figure}

\begin{figure}[htbp]
    \centering
    \includegraphics[width=1.0\textwidth]{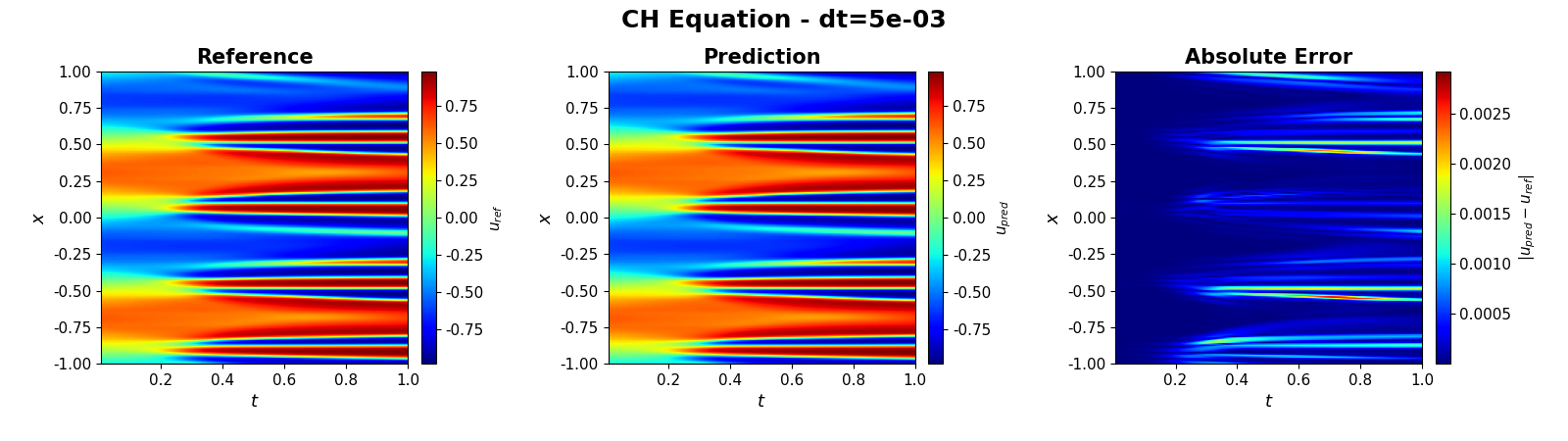}
    \caption{1D Cahn--Hilliard equation with $\Delta t=5\times10^{-3}$: reference solution (left), predicted solution (center), and absolute error (right).}
\end{figure}

\begin{figure}[htbp]
    \centering
    \includegraphics[width=1.0\textwidth]{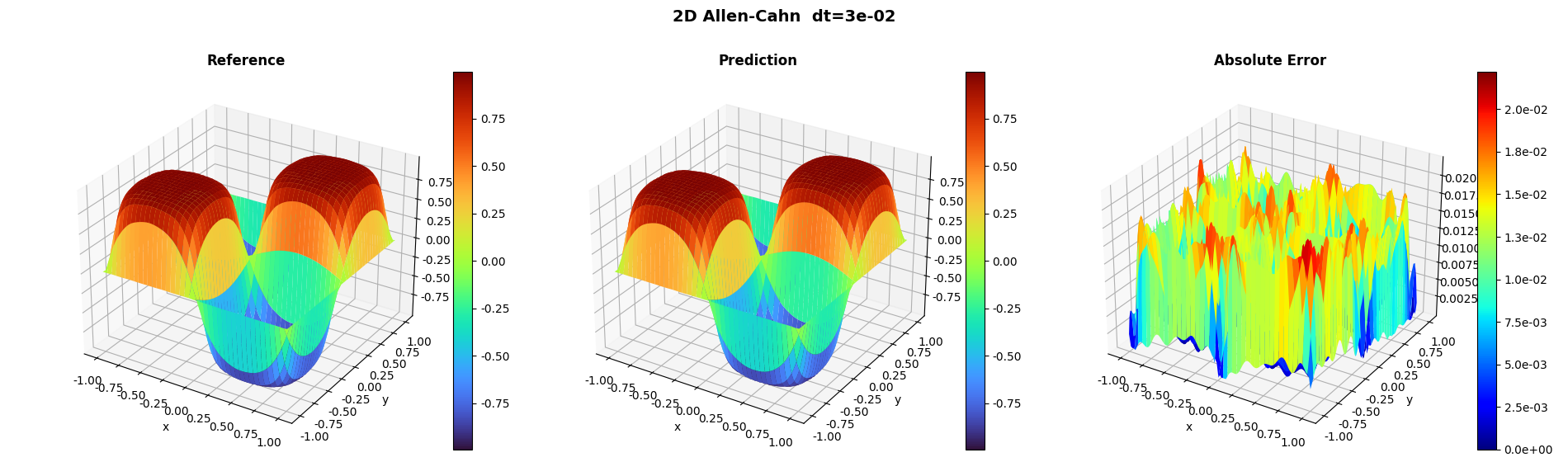}
    \caption{2D Allen--Cahn equation with $\Delta t=3\times10^{-2}$: reference solution (left), predicted solution (center), and absolute error (right).}
\end{figure}

\begin{figure}[htbp]
    \centering
    \includegraphics[width=1.0\textwidth]{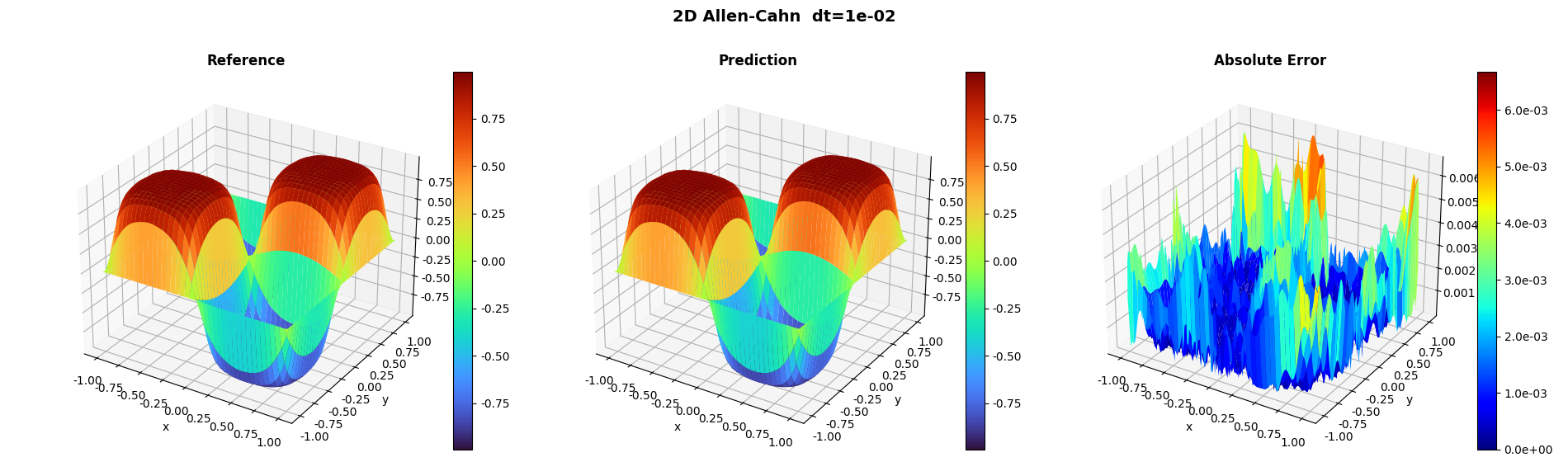}
    \caption{2D Allen--Cahn equation with $\Delta t=10^{-2}$: reference solution (left), predicted solution (center), and absolute error (right).}
\end{figure}

\bibliographystyle{plain}
\bibliography{bib.bib}

\end{document}